\documentclass[onefignum,onetabnum]{siamart190516}

\usepackage{lipsum}
\usepackage{amsfonts}
\usepackage{graphicx}
\usepackage{epstopdf}
\usepackage{algorithmic}
\usepackage{subcaption}
\usepackage{multirow}

\ifpdf
  \DeclareGraphicsExtensions{.eps,.pdf,.png,.jpg}
\else
  \DeclareGraphicsExtensions{.eps}
\fi

\newsiamremark{assumption}{Assumption}
\newsiamremark{remark}{Remark}
\newsiamremark{hypothesis}{Hypothesis}
\crefname{hypothesis}{Hypothesis}{Hypotheses}
\newsiamthm{claim}{Claim}

\headers{Stein variational reduced basis Bayesian inversion}{P. Chen and O. Ghattas}

\title{Stein variational reduced basis Bayesian inversion}

\author{Peng Chen \thanks{Oden Institute for Computational Science and Engineering, The University of Texas at Austin, Austin, TX 78712.  (\email{peng@oden.utexas.edu})} \and Omar Ghattas \thanks{Department of
		Mechanical Engineering, and Department of Geological Sciences, Oden Institute
		for Computational Engineering \& Sciences, The
		University of Texas at Austin, Austin, TX 78712 (\email{omar@oden.utexas.edu})}}

\usepackage{amsopn}

\newcommand{\bA}{{\mathbb{A}}}
\newcommand{\bE}{{\mathbb{E}}}

\newcommand{\bN}{{\mathbb{N}}}
\newcommand{\bR}{{\mathbb{R}}}
\newcommand{\bO}{{\mathbb{O}}}

\newcommand{\cH}{\mathcal{H}}

\newcommand{\cL}{\mathcal{L}}
\newcommand{\cN}{\mathcal{N}}
\newcommand{\cO}{\mathcal{O}}
\newcommand{\cS}{\mathcal{S}}
\newcommand{\cT}{\mathcal{T}}

\newcommand{\bsb}{\boldsymbol{b}}

\newcommand{\bsf}{\boldsymbol{f}}

\newcommand{\bsu}{\boldsymbol{u}}
\newcommand{\bspsi}{\boldsymbol{\psi}}

\newcommand{\beq}{\begin{equation}}
\newcommand{\eeq}{\end{equation}}

\DeclareMathOperator*{\argmax}{arg\,max}

\ifpdf
\hypersetup{
  pdftitle={Stein variational reduced basis Bayesian inversion},
  pdfauthor={P. Chen and O. Ghattas}{\tiny }
}
\fi

\begin{document}

\maketitle

\begin{abstract}
We propose and analyze a Stein variational reduced basis method (SVRB) to solve large-scale PDE-constrained Bayesian inverse problems. To address the computational challenge of drawing numerous samples requiring expensive PDE solves from the posterior distribution, we integrate an adaptive and goal-oriented model reduction technique with an optimization-based Stein variational gradient descent method (SVGD). The samples are drawn from the prior distribution and iteratively pushed to the posterior by a sequence of transport maps, which are constructed by SVGD, requiring the evaluation of the potential---the negative log of the likelihood function---and its gradient with respect to the random parameters, which depend on the solution of the PDE. To reduce the computational cost, we develop an adaptive and goal-oriented model reduction technique based on reduced basis approximations for the evaluation of the potential and its gradient. We present a detailed analysis for the reduced basis approximation errors of the potential and its gradient, the induced errors of the posterior distribution measured by Kullback--Leibler divergence, as well as the errors of the samples. To demonstrate the computational accuracy and efficiency of SVRB, we report results of numerical experiments on a Bayesian inverse problem governed by a diffusion PDE with random parameters with both uniform and Gaussian prior distributions. Over 100X speedups can be achieved while the accuracy of the approximation of the potential and its gradient is preserved.

\end{abstract}




\section{Introduction}

Uncertainty, often represented in the form of random parameters, is ubiquitous in computational modeling and simulation of scientific, engineering, and societal systems. Quantifying the input uncertain is critical for making reliable system predictions and their optimization. Given possibly noisy observational data on some system outputs, Bayesian inversion provides a versatile and optimal approach for inference of the random parameters---versatile in modeling the random parameters by prior distributions in a probability framework, and optimal in providing the posterior distribution of the random parameters that optimally matches the system output and the observational data in a certain entropy measure \cite{Stuart10}. Drawing samples from the posterior distribution, and evaluating some statistics, e.g., expectation, variance, failure probability, etc., of a given quantity of interest, are among the central tasks of Bayesian inversion. These tasks are often computationally challenging due to a number of factors. First, the geometry of the posterior distribution may have complex features in the parameter space, such as local concentration, multimodality, and non-Gaussianity. Second, sampling from the posterior distribution is a significant challenging when the parameter space is high-dimensional due to the curse of dimensionality faced by many computational methods, i.e., the complexity grows exponentially with respect to the parameter dimension. Third, each evaluation of the system output involves solving a model that describes the system, which may be very expensive and makes sampling methods that require numerous evaluations of the output prohibitive. In particular, we consider systems modeled by partial differential equations (PDE) that are computationally expensive to solve.


To address the above challenges related to sampling from the posterior distribution, many computational methods have been developed \cite{Stuart10, BieglerBirosGhattasEtAl11}, among which we mention the following developments over the last decade. Built on the classical Markov chain Monte Carlo (MCMC) method \cite{Hastings70}, the geometry of the (log) posterior as captured by its gradient, Hessian, and higher order derivatives with respect to (w.r.t.) the parameter has been exploited to accelerate the convergence of MCMC \cite{GirolamiCalderhead11, MartinWilcoxBursteddeEtAl12, HoangSchwabStuart13, Bui-ThanhGhattasMartinEtAl13, SpantiniSolonenCuiEtAl15,  IsaacPetraStadlerEtAl15,  PetraMartinStadlerEtAl14, BeskosGirolamiLanEtAl17}. By taking advantage of the smoothness and sparsity of the posterior w.r.t.\ the parameters, sparse polynomial approximations and sparse quadratures based on sparse grids have been developed to achieve fast convergence in evaluating various statistics w.r.t.\ the posterior \cite{MarzoukXiu09, SchwabStuart12, SchillingsSchwab13, ChenSchwab15, ChenSchwab16a, ChenSchwab16b, ChenSchwab16c, SchillingsSchwab16,  ChenVillaGhattas17, ZahmCuiLawEtAl18, ZahrCarlbergKouri19, FarcasLatzUllmannEtAl19}. Transport-based variational methods have been recently developed that push the prior samples to the posterior through a measure transport, which is obtained by solving an optimization problem \cite{ElMoselhyMarzouk12, LiuWang16, Liu17, LiuZhu18, DetommasoCuiMarzoukEtAl18, SpantiniBigoniMarzouk18, ChenMackeyGorhamEtAl18, LuLuNolen19, ChenWuChenEtAl19, LiLiuLiuEtAl19}. To reduce the expensive computational cost in solving the PDE model, various types of surrogate models are constructed to approximate the PDE solution and the posterior with significantly reduced expense. In particular we mention projection based reduced order models 	\cite{NguyenRozzaHuynhEtAl09, LiebermanWillcoxGhattas10, CuiMarzoukWillcox15, ChenSchwab16a, ChenSchwab16b, ChenSchwab16c, CuiMarzoukWillcox16, ChenQuarteroniRozza17, ZahrCarlbergKouri19}.

In this work, we propose and analyze a new computational approach for large-scale PDE-constrained Bayesian inverse problems by integrating an adaptive and goal-oriented model reduction technique with a Stein variational gradient descent method (SVGD). To draw samples from the posterior distribution, we employ the optimization-based SVGD to iteratively transport samples drawn from the prior distribution to the posterior. This approach requires the evaluation of a potential---the negative log of the likelihood function---and its gradient with respect to the random parameters. In each SVGD iteration, to reduce the computational cost in solving the large-scale PDEs, we develop an adaptive and goal-oriented model reduction technique based on reduced basis approximations for the evaluation of the potential and its gradient. More specifically, to achieve accurate and efficient reduced basis approximation, we use a dual-weighted residual as the a-posteriori error indicator for the potential and propose an adaptive greedy algorithm to construct the reduced basis spaces for 
both a state PDE and an adjoint PDE. Moreover, we use the dual-weighted residual and its gradient as correctors to improve the approximation accuracy of the potential and its gradient, which involves solution of an incremental state PDE and an incremental adjoint PDE by reduced basis approximations. Furthermore, to guarantee certified reduced basis approximations at all SVGD samples, especially for the SVGD samples driven close to the posterior distribution, we take the SVGD samples as the training samples and decrease the tolerance in the greedy construction informed by a convergence criterion of the SVGD iteration. We present a detailed analysis for the reduced basis approximation errors of the potential and its gradient, the induced errors of the posterior distribution measured by Kullback--Leibler divergence, as well as the errors for the samples. To demonstrate the computational accuracy and efficiency of the proposed Stein variational reduced basis (SVRB) method, we report results of numerical experiments on PDE-based Bayesian inference for random parameters with both uniform and Gaussian distributions.  

The rest of the paper is organized as follows. In Section \ref{sec:Bayesian} we present the general formulation of Bayesian inverse problems, followed by Section \ref{sec:Stein} on the SVGD method to draw samples from the posterior distribution. Section \ref{sec:modelreduction} is devoted to the development of the goal-oriented and adaptive reduced basis method, whose accuracy and efficiency are analyzed in Section \ref{sec:errorestimates} and demonstrated by numerical experiments in Section \ref{sec:numerics}. Conclusions and perspectives are provided in Section \ref{sec:conclusion}.

\section{Bayesian inversion}
\label{sec:Bayesian}

Let $\theta = (\theta_1, \dots, \theta_d) \in \Theta \subset \bR^d$ denote a vector of random parameters defined in the parameter space $\Theta$ of dimension $d\in \bN$, which is assumed to have a prior distribution with density function $p_0: \Theta \to \bR$. Let $f: \bR^d \to \bR^s$ denote a parameter-to-observable map with observational data $y \in \bR^s$ of dimension $s\in \bN$ given by 
\begin{equation}
y = f(\theta) + \xi,
\end{equation}
where $\xi$ represents an observation noise. We assume the noise has Gaussian distribution $\cN(0, \Gamma)$ with  symmetric positive definite covariance $\Gamma \in \bR^{s\times s}$. Under the assumption that $\theta$ and $\xi$ are independent, Bayes' rule provides the posterior distribution with density function $p_y:\Theta \to \bR$ as
\begin{equation}\label{eq:posterior_density}
p_y(\theta) = \frac{1}{Z} p(y|\theta) p_0(\theta),
\end{equation}
where $p(y|\theta)$ is a likelihood function given by 
\begin{equation}
p(y|\theta) = \exp\left(- \eta_y(\theta)) \right),
\end{equation}
with the potential function $\eta_y:\Theta \to \bR$ defined as
\begin{equation}\label{eq:potential}
\eta_y(\theta) = \frac{1}{2} ||y - f(\theta)||_\Gamma = \frac{1}{2} (y - f(\theta))^T \Gamma^{-1} (y - f(\theta)).
\end{equation}
$Z$ represents a normalization constant, given by 
\begin{equation}
Z = \int_\Theta p(y|\theta) p_0(\theta) d\theta,
\end{equation}
which is often computational intractable, especially for large dimension $d$.
The central task of Bayesian inversion is to sample from the posterior distribution and compute some statistics of a given quantity of interest w.r.t.\ the posterior, e.g., mean, variance, failure probability, etc. Challenges arise when (1) the posterior distribution has complex geometry, e.g., concentrating in a local parameter region, featuring multiple modes; (2) the parameter-to-observable map $f$ is very expensive to evaluation, e.g., it involves large-scale PDE solve; (3) the parameter dimension $d$ is high. 

\section{Stein variational gradient descent}
\label{sec:Stein}

To draw samples from the posterior distribution $\mu_y$ with density $p_y$, we seek an invertible transport map $T:\bR^d \to \bR^d$ such that for any sample $\theta$ drawn from the prior distribution $\mu_0$ with density $p_0$, $T(\theta)$ is a sample drawn from the posterior distribution, or equivalently we seek $T$ such that $T_\sharp \mu_0 = \mu_y$, where $T_\sharp$ represents a pushforward map that satisfies 
\begin{eqnarray}
p_y(\theta) = p_0(T^{-1}(\theta)) |\text{det} \nabla_\theta T^{-1}(\theta)| \text{ or } p_0(\theta) = p_y(T(\theta)) |\text{det} \nabla_\theta T(\theta)|,
\end{eqnarray}
where $\text{det}$ represents the determinant of a matrix. 
 For this purpose, one common practice is to find the transport map $T$ in a certain function class $\cT$ that minimizes the Kullback--Leibler (KL) divergence between $T_\sharp \mu_0$ and $\mu_y$, i.e.,
\begin{equation}\label{eq:DKL}
\min_{T \in \cT} D_{\text{KL}}(T_\sharp \mu_0 | \mu_y),
\end{equation}
where the KL divergence between two probability distribution $\mu_1, \mu_2$ with densities $p_1, p_2$, which satisfy the absolute continuity $\mu_1 \ll \mu_2$, i.e., $d\mu_1(\theta)/d\mu_2(\theta) = p_1(\theta)/p_2(\theta) \geq 0$ for all $\theta \in \Theta$, is defined as 
\begin{equation}\label{eq:KL-divergence}
D_{\text{KL}}(\mu_1 | \mu_2) := \bE_{\theta \sim \mu_1} \left[\log\frac{d\mu_1}{d\mu_2}\right] = \int_\Theta p_1(\theta) \log\frac{p_1(\theta)}{p_2(\theta)} d\theta.
\end{equation}

To solve the optimization problem \eqref{eq:DKL}, we first form the (possibly very complex) transport map $T$ as a composition of a sequence of (much simpler) transport maps as 
\begin{equation}
T = T_L \circ T_{L-1} \circ \cdots \circ T_1 \circ T_0,
\end{equation}
where $T_l$, $l = 0, \dots, L$, are invertible perturbation maps from identity defined as
\begin{equation}
T_l := I + \alpha_l Q_l, \quad \text{i.e., } T_l(\theta) := \theta + \alpha_l Q_l(\theta), \quad \forall \theta \in \Theta,
\end{equation}
where $\alpha_l$ is a step size, $I$ is the identity map, $Q_l: \bR^d \to \bR^d$ is a perturbation map. For $l = 1, \dots, L$, let $\mu_l$ denote the probability distribution defined as  
\begin{equation}
\mu_{l} = (T_{l-1} \circ \cdots \circ T_0)_\sharp \mu_0.
\end{equation}
By a gradient descent method, $Q_l$ is taken as the negative of the first order variation of $D_{\text{KL}} ((I + Q)_\sharp \mu_{l} | \mu_y)$ w.r.t.\ $Q$, evaluated at $Q=0$, i.e.,
\begin{equation}
Q_l := - \nabla_{Q} D_{\text{KL}} ((I + Q)_\sharp \mu_{l} | \mu_y) |_{Q=0}.
\end{equation}


Let $\cH$ denote a reproducing kernel Hilbert space (RKHS) with reproducing kernel $k(\cdot, \cdot): \bR^d \times \bR^d \to \bR$, e.g., a radial basis function kernel \cite{LiuWang16}
\begin{equation}\label{eq:kernel}
k(\theta, \theta') = \exp\left(- \frac{1}{h} ||\theta - \theta'||_2^2\right),
\end{equation}
for a suitable scaling factor $h > 0$, where $||\cdot||_2$ denotes the Euclidean norm.
By taking the function space as $\cT = \cH^d$, the tensor product of $\cH$,
we obtain \cite{LiuWang16}
\begin{equation}\label{eq:Q_l}
Q_l(\cdot) = \bE_{\theta \sim \mu_{l}} \left[
\cS_{p_y(\theta)} \otimes k(\theta, \cdot)
\right],
\end{equation}
where the expectation is take w.r.t.\ the distribution $\mu_l$, $\cS_{p_y(\theta)}$ is the Stein operator, 
\begin{equation}\label{eq:Stein}
\cS_{p_y(\theta)} \otimes k(\theta, \cdot) = \nabla_\theta \log(p_y(\theta)) k(\theta, \cdot) + \nabla_\theta k(\theta, \cdot),
\end{equation}
for which the gradient descent method with the perturbation map $Q_l$ given by the first order variation in \eqref{eq:Q_l} is named Stein variational gradient descent (SVGD) \cite{LiuWang16}. Note that $p_y(\theta)$ in \eqref{eq:Stein} is the posterior density defined in \eqref{eq:posterior_density}, which involves the normalization constant $Z$ that is often computationally intractable. However, the gradient $\nabla_\theta \log(p_y(\theta))$ in \eqref{eq:Stein} does not depend on $Z$ since, by definition \eqref{eq:posterior_density}, 
\begin{equation}\label{eq:gradient_logposterior}
\nabla_\theta \log(p_y(\theta)) = \frac{\nabla_\theta (p(y|\theta)p_0(\theta))}{p(y|\theta)p_0(\theta)} = -\nabla_\theta \eta_y(\theta) + \frac{\nabla_\theta p_0(\theta)}{p_0(\theta)}.
\end{equation}
 To evaluate the expectation in \eqref{eq:Q_l} w.r.t.\ $\mu_{l}$, a sample average approximation is used,
\begin{equation}\label{eq:pertubation}
Q_l (\cdot) \approx  \hat{Q}_l(\cdot) =
\frac{1}{M} \sum_{m = 1}^M \left(\nabla_{\theta_m^{l}} \log(p_y(\theta_m^{l})) k(\theta_m^{l}, \cdot) + \nabla_{\theta_m^{l}} k(\theta_m^{l}, \cdot)\right),
\end{equation}
where the samples $\theta_m^l$, $m = 1, \dots, M$, $l = 1, \dots, L$, are given by 
\begin{equation}\label{eq:samples}
\theta_m^l = \theta_m^{l-1} + \alpha_l \hat{Q}_l(\theta_m^{l-1}), 
\end{equation}
with the initial samples $\theta_m^0$, $m = 1, \dots, M$, randomly drawn from the prior distribution $\mu_0$.
Convergence of the empirical distribution of the samples $\theta_m^L$, $m = 1, \dots, M$, to the posterior distribution $\mu_y$ when $M, L \to \infty$ was established in \cite{Liu17} for the step size $\alpha_l$ satisfying suitable conditions. 
The SVGD is summarized in Algorithm \ref{alg:SVGD}. 

\begin{algorithm} 
\caption{Stein variational gradient descent (SVGD)} 
\label{alg:SVGD} 
\begin{algorithmic}[1] 
\STATE \textbf{Input:} random samples $\theta_m^0 \sim \mu_0$, $m = 1, \dots, M$, tolerance $\varepsilon$, maximum step $L$.
\STATE \textbf{Output:} samples $\theta_m$, $m = 1, \dots, M$, approximate of the posterior $\mu_y$.
\STATE Initialize $l = 0$, $t_l = 2\varepsilon$.
    \WHILE{$l \leq L$ and $t_l > \varepsilon$}\label{line:stop}
    \STATE For each $n = 1, \dots, M$, compute
    $$
     \hat{Q}_l(\theta_n^{l}) = \frac{1}{M} \sum_{m = 1}^M \left(\nabla_{\theta_m^{l}} \log(p_y(\theta_m^{l})) k(\theta_m^{l}, \theta_n^l) + \nabla_{\theta_m^{l}} k(\theta_m^{l}, \theta_n^l)\right).
    $$
    \STATE Compute a step size $\alpha_l$.
    \STATE For each $n = 1, \dots, M$, update the samples 
    $$
    \theta_n^{l+1} = \theta_n^{l} + \alpha_l \hat{Q}_l(\theta_n^l).
    $$
    \STATE Compute a stopping indicator $t_l$ and update $l = l+1$.
    \ENDWHILE
    \STATE Set $\theta_m = \theta_m^l$, $m = 1, \dots, M$.
\end{algorithmic}
\end{algorithm}

The indicator $t_l$ in line \ref{line:stop} is used as one stopping criterion. Here we use
\begin{equation}\label{eq:t-l}
t_l = \max_{m = 1, \dots, M} ||\hat{Q}_l(\theta_m^{l-1})||_2,
\end{equation}
i.e., the algorithm stops when the maximum gradient norm becomes smaller than a given tolerance. Different adaptive methods can be used to compute the step size $\alpha_l$. Here we use a line search method to find $\alpha_l$ such that 
the KL divergence $D_{\text{KL}}((T_l)_\sharp \mu_l | \mu_y)$ is reduced, which is given by
\begin{equation}
D_{\text{KL}}((T_l)_\sharp \mu_l | \mu_y) = D_{\text{KL}}(\mu_l | (T_l)^\sharp\mu_y) = \bE_{\theta \sim \mu_l} \left(\log \frac{ p_l(\theta)}{p_y(T_l(\theta)) |\det \nabla T_l(\theta)|}\right),
\end{equation}
where $T^\sharp_l$ is the puallback such that $(T_l)^\sharp\mu_y(d\theta) = p_y(T_l(\theta)) |\det \nabla T_l(\theta)| d\theta$.
As $p_l(\theta)$ does not depend on $\alpha_l$, we define a merit function for the line search by sample average approximation of the expectation as 
\begin{equation}
 \bE_{\theta \sim \mu_l} \left(\log \frac{ p_l(\theta)}{p_y(T_l(\theta)) |\det \nabla T_l(\theta)|}\right) \approx -\frac{1}{M} \sum_{m=1}^M \log(p_y(T_l(\theta_m^l))) + \log(|\det \nabla T_l(\theta_m^l)|).
\end{equation}
In practice, we can neglect the second term, which is close to zero when the transport map is close to identity, especially towards the convergence.

\section{Model reduction}
\label{sec:modelreduction}

In this work, we consider that the evaluation of the parameter-to-observable map $f$ involves solution of parametric partial differential equations (PDE), which is computationally expensive to solve. In this section, we develop a goal-oriented model reduction technique based on reduced basis approximation for its evaluation as well as its gradient required by SVGD. 

\subsection{Parametric PDE models}
We consider that the parameter-to-observable map $f$ is given by 
\begin{equation}\label{eq:p-t-o_map}
f(\theta) = \cO(u(\theta)),
\end{equation}
where 
$\cO = (o_1, \dots, o_s): V \to \bR^s$ denotes a vector of observation functionals as a mapping from a Hilbert space $V$ to $\bR^s$; $u(\theta)$ is the solution of a linear parametric PDE model, given in weak form as: given $\theta \in \Theta$, find $u(\theta) \in V$ such that 
\begin{equation}\label{eq:state}
A(u, v; \theta) = F(v; \theta), \quad \forall v \in V,
\end{equation}
where $v \in V$ is a test variable. $A(\cdot, \cdot; \theta): V \times V \to \bR$ is a parametric bilinear form. $F(\cdot; \theta): V \to \bR$ is a parametric linear form. 

To evaluate the gradient of the log-posterior \eqref{eq:gradient_logposterior}, we need to compute the gradient of the potential $\nabla_\theta \eta_y(\theta)$, or more explicitly $\nabla_\theta \eta_y(u(\theta))$ since $\eta_y$ depends on $\theta$ through the solution $u$ of problem \eqref{eq:state}. To compute the gradient, we apply a Lagrange multiplier approach by defining the Lagrangian 
\begin{equation}
\cL(u, \psi, \theta) := \eta_y(u) + A(u, \psi; \theta) - F(\psi; \theta),
\end{equation}
where $\psi$ is an adjoint variable or a Lagrange multiplier. By taking the variation of the Lagrangian w.r.t.\ $u$ as zero, we obtain: find $\psi \in V$ such that 
\begin{equation}\label{eq:adjoint}
A(w, \psi; \theta) = -\nabla_u \eta_y|_u(w), \quad \forall w \in V,
\end{equation}
where by the definition of the potential $\eta_y$ in \eqref{eq:potential} we have
\begin{equation}\label{eq:nabla_u-eta}
\nabla_u \eta_y|_u(w) = - (y - \cO(u))^T \Gamma^{-1} \cO(w).
\end{equation}
With the state $u(\theta)$ and adjoint $\psi(\theta)$ obtained, we can evaluate $\nabla_\theta \eta_y(u)$ by 
\begin{equation}\label{eq:eta-gradient}
\nabla_\theta \eta_y(u) = \partial_\theta \cL(u, \psi, \theta) = \partial_\theta A(u, \psi; \theta) - \partial_\theta F(\psi; \theta).
\end{equation}

\subsection{High-fidelity approximations}
To numerically solve the parametric PDE \eqref{eq:potential} and its adjoint \eqref{eq:adjoint}, and to evaluate the potential and its gradient, we apply high-fidelity approximations based on a finite element method. 

Let $V_h \subset V$ denote a finite element subspace of $V$ with degrees of freedom (DOF) $N_h$.
Let $\{\phi_h^n, n = 1, \dots, N_h\}$ denote the (piecewise polynomial) basis of $V_h$. Then the high-fidelity approximation of the parametric PDE \eqref{eq:potential} reads: given $\theta \in \Theta$, find $u_h \in V_h$ such that 
\begin{equation}\label{eq:state-HiFi}
A(u_h, v_h; \theta) = F(v_h; \theta), \quad \forall v_h \in V_h.
\end{equation}


By writting the representation of $u_h$ in $V_h$ as
\begin{equation}
u_h = \sum_{n=1}^{N_h} u_h^n \phi_h^n,
\end{equation}
where $\bsu_h = (u_h^1, \dots, u_h^{N_h})^T \in \bR^{N_h}$ is a coefficient vector, we obtain the algebraic system corresponding to \eqref{eq:state-HiFi} as 
\begin{equation}\label{eq:state-algebraic}
\bA_h(\theta) \bsu_h = \bsf_h(\theta),
\end{equation}
where the parametric matrix $\bA_h(\theta)$ and vector $\bsf_h(\theta)$ are given by 
\begin{equation}
(\bA_h(\theta))_{m,n} = A(\phi_h^n, \phi_h^m; \theta) \; \text{ and } \; (\bsf_h(\theta))_{m} = F(\phi_h^m; \theta), \quad m, n = 1, \dots, N_h.
\end{equation}

Similarly, the high-fidelity approximation of the adjoint PDE \eqref{eq:adjoint} reads: given $\theta \in \Theta$, and $u_h$ as the solution of \eqref{eq:state-HiFi} at $\theta$, find $\psi_h \in V_h$ such that 
\begin{equation}\label{eq:adjoint-HiFi}
A(w_h, \psi_h; \theta) = -\nabla_u \eta_y|_{u_h}(w_h), \quad \forall w_h \in V_h.
\end{equation}

We write the representation of $\psi_h$ in $V_h$ as 
\begin{equation}
\psi_h = \sum_{n=1}^{N_h} \psi_h^n \phi_h^n,
\end{equation}
where the coefficient vector $\bspsi_h = (\psi_h^1, \dots, \psi_h^{N_h})^T \in \bR^{N_h}$ is the solution of the algebraic system 
\begin{equation}\label{eq:adjoint-algebraic}
\bA^T_h(\theta) \bspsi_h = \bsb^u_h,
\end{equation}
where $\bA^T_h(\theta)$ is the transpose of $\bA_h(\theta)$. The vector $\bsb^u_h$ is given by 
\begin{equation}
(\bsb^u_h)_m = (y - \cO(u_h))^T \Gamma^{-1} \cO(\phi_h^m) = (y - \bsu_h^T \bO_h)^T \Gamma^{-1} \bO^m_h, \quad m = 1, \dots, N_h,
\end{equation}
where the matrix $\bO_h \in \bR^{N_h\times s}$, whose $m$-th row $\bO^m_h \in \bR^s$ is given by 
\begin{equation}
\bO^m_h = (o_1(\phi_h^m), \dots, o_s(\phi_h^m))^T, \quad m = 1, \dots, N_h.
\end{equation}
At any $\theta \in \Theta$, with the high-fidelity state $u_h$ and adjoint $\psi_h$ obtained as the solutions of \eqref{eq:state-HiFi} and \eqref{eq:adjoint-HiFi}, we can approximate the potential $\eta_y$ as
\begin{equation}\label{eq:eta-HiFi}
\eta_y(u_h) = \frac{1}{2}(y - \cO(u_h))^T \Gamma^{-1} (y - \cO(u_h)) = \frac{1}{2} (y - \bsu_h^T \bO_h)^T \Gamma^{-1} (y - \bsu_h^T \bO_h),
\end{equation}
and its gradient by \eqref{eq:eta-gradient} as
\begin{equation}\label{eq:eta-gradient-HiFi}
\nabla_\theta \eta_y(u_h) = \partial_\theta A(u_h, \psi_h; \theta) - \partial_\theta F(\psi_h; \theta) = \bspsi_h^T \partial_\theta \bA_h(\theta) \bsu_h - \bspsi_h^T \partial_\theta \bsf_h(\theta).
\end{equation}

\subsection{Reduced basis approximations}
\label{sec:RB}
When the DOF $N_h$ of the high-fidelity approximations based on the finite element method is large, which is often required for the approximations to achieve accuracy, solution of the systems \eqref{eq:state-algebraic} and \eqref{eq:adjoint-algebraic} are computationally very expensive. To reduce the computational cost, we employ reduced basis approximations based on a reduced basis method \cite{HesthavenRozzaStamm15, QuarteroniManzoniNegri2015}, which is one common model reduction technique built on the high-fidelity approximations. 

Let $V_r \subset V_h$ denote a reduced basis space spanned by $N^u_r \ll N_h$ basis functions $\{\phi_u^n \in V_h, n = 1, \dots, N^u_r\}$, which will be constructed later for the solution manifold of the parametric PDE \eqref{eq:state-HiFi}. Then the reduced basis approximation of the parametric PDE \eqref{eq:state} reads: given $\theta \in \Theta$, find $u_r \in V_r$ such that  
\begin{equation}\label{eq:state-RB}
A(u_r, v_r; \theta) = F(v_r; \theta), \quad \forall v_r \in V_r.
\end{equation}

Similar to the high-fidelity expansion, we write the representation of $u_r$ in $V_r$ as 
\begin{equation}
u_r = \sum_{n = 1}^{N_r^u} u_r^n \phi_u^n,
\end{equation}
where $\bsu_r = (u_r^1, \dots. u_r^{N^u_r}) \in \bR^{N^u_r}$ is the coefficient vector of $u_r$, which is a solution of the algebraic system 
\begin{equation}
\bA^u_r(\theta) \bsu_r = \bsf^u_r(\theta).
\end{equation}
The parametric matrix $\bA^u_r(\theta)$ and vector $\bsf^u_r(\theta)$ are given by 
\begin{equation}
(\bA^u_r(\theta))_{m,n} = A(\phi_u^n, \phi_u^m; \theta) \; \text{ and } \; (\bsf^u_r(\theta))_{m} = F(\phi_u^m; \theta), \quad m, n = 1, \dots, N^u_r.
\end{equation}

Analogously, let $W_r \subset V_h$ denote a reduced basis space spanned by $N^\psi_r \ll N_h$ basis functions $\{\phi_\psi^n \in V_h, n = 1, \dots, N^\psi_r\}$, which will be constructed later for the adjoint solution manifold of the parametric adjoint PDE \eqref{eq:adjoint-HiFi}. Then the reduced basis approximation of the parametric adjoint PDE \eqref{eq:adjoint} reads: given $\theta \in \Theta$, and the solution $u_r \in V_r$ of problem \eqref{eq:state-RB} at $\theta$, find $\psi_r \in W_r$ such that
\begin{equation}\label{eq:adjoint-RB}
A(w_r, \psi_r; \theta) = -\nabla_u \eta_y|_{u_r}(w_r), \quad \forall w_r \in W_r.
\end{equation}

We write the representation of $\psi_r$ in $W_r$ as 
\begin{equation}
\psi_r = \sum_{n=1}^{N^\psi_r} \psi_r^n \phi_\psi^n,
\end{equation}
where $\bspsi_r = (\psi_r^1, \dots, \psi_r^{N^\psi_r})^T \in \bR^{N_r}$ is the solution of the algebraic system 
\begin{equation}\label{eq:adjoint-RB-algebraic}
\bA^\psi_r(\theta) \bspsi_r = \bsb^\psi_r,
\end{equation}
where the parametric matrix $\bA^\psi_r(\theta)$ is given by
\begin{equation}\label{eq:A-r-psi}
(\bA^\psi_r(\theta))_{m,n} = A(\phi_\psi^m, \phi_\psi^n; \theta), \quad m, n = 1, \dots, N^\psi_r.
\end{equation}
The vector $\bsb^\psi_r$ is given by
\begin{equation}
(\bsb^\psi_r)_m = (y - \cO(u_r))^T \Gamma^{-1} \cO(\phi_\psi^m) = (y - \bsu_r^T \bO_r^u)^T \Gamma^{-1} (\bO_r^\psi)^m, \quad m = 1, \dots, N^\psi_r,
\end{equation}
where the matrix $\bO^u_r \in \bR^{N^u_r\times s}$, whose $m$-th row $(\bO^u_r)^m \in \bR^s$ is given by 
\begin{equation}
(\bO^u_r)^m = (o_1(\phi_u^m), \dots, o_s(\phi_u^m))^T, \quad m = 1, \dots, N^u_r,
\end{equation}
and the matrix $\bO^\psi_r \in \bR^{N^\psi_r\times s}$, whose $m$-th row $(\bO^\psi_r)^m \in \bR^s$ is given by 
\begin{equation}
(\bO^\psi_r)^m = (o_1(\phi_\psi^m), \dots, o_s(\phi_\psi^m))^T, \quad m = 1, \dots, N^\psi_r.
\end{equation}
At any $\theta \in \Theta$, with the reduced basis state $u_r$ and adjoint $\psi_r$ obtained as the solutions of \eqref{eq:state-RB} and \eqref{eq:adjoint-RB}, we can approximate the potential $\eta_y$ as 
\begin{equation}\label{eq:eta-RB}
\eta_y(u_r) = \frac{1}{2}(y - \cO(u_r))^T \Gamma^{-1} (y - \cO(u_r)) = \frac{1}{2} (y - \bsu_r^T \bO_r^u)^T \Gamma^{-1} (y - \bsu_r^T \bO_r^u),
\end{equation}
and its gradient by \eqref{eq:eta-gradient} as
\begin{equation}\label{eq:eta-gradient-RB}
\nabla_\theta \eta_y(u_r) = \partial_\theta A(u_r, \psi_r; \theta) - \partial_\theta F(\psi_r; \theta) = \bspsi_r^T \partial_\theta \bA_r^{u,\psi}(\theta) \bsu_r - \bspsi_r^T \partial_\theta \bsf^\psi_r(\theta),
\end{equation}
where the matrix $\bA_r^{u,\psi}(\theta) \in \bR^{N_r^\psi \times N_r^u}$, whose $mn$-th element for $ m = 1, \dots, N^\psi_r, n = 1, \dots, N^u_r$, is given by 
\begin{equation}\label{eq:Arupsi}
(\bA_r^{u,\psi}(\theta))_{m,n} = A(\phi_u^n, \phi_\psi^m; \theta).
\end{equation}
The vector $\bsf^\psi_r(\theta) \in \bR^{N_r^\psi}$, whose $m$-th element for $m =1 , \dots, N_r^\psi$, is given by 
\begin{equation}\label{eq:frpsi}
(\bsf^\psi_r(\theta))_m = F(\phi_\psi^m; \theta).
\end{equation}

\subsection{Goal-oriented approximations}

To develop a goal-oriented construction of the reduced basis functions $\phi_u^m$ and $\phi_\psi^m$, $m = 1, \dots, N_r$, and to improve the accuracy of the reduced basis approximations of the potential $\eta_y(u(\theta))$ and its gradient $\nabla_\theta \eta_y(u(\theta))$, we use a dual weighted residual, as employed in \cite{ChenSchwab15, ChenSchwab16a} to approximate the potential $\eta_y(u(\theta))$ for Bayesian inversion.
At any $\theta \in \Theta$, with the state $u_r$ and adjoint $\psi_r$ obtained as the solutions of the reduced parametric problems \eqref{eq:state-RB} and \eqref{eq:adjoint-RB}, we define the dual weighted residual as 
\begin{equation}\label{eq:Delta}
\Delta^\eta_r(\theta) := A(u_r, \psi_r; \theta) - F(\psi_r;\theta),
\end{equation}
which can be evaluated as 
\begin{equation}
\Delta^\eta_r(\theta) = \bspsi_r^T A_r^{u, \psi}(\theta) \bsu_r - \bspsi_r^T \bsf_r^\psi(\theta),
\end{equation}
where $A_r^{u, \psi}$ and $\bsf_r^\psi$ are given in \eqref{eq:Arupsi} and \eqref{eq:frpsi}. 

Based on the dual weighted residual $\triangle_r^\eta(\theta)$, we define a modified reduced basis approximation of the potential as 
\begin{equation}\label{eq:dwreta}
\eta_y^\Delta(\theta) := \eta_y(u_r(\theta)) + \Delta_r^\eta(\theta).
\end{equation}
To compute $\partial_\theta \eta_y^\Delta(\theta)$, the gradient of the reduced basis approximation of the potential with modification in \eqref{eq:dwreta}, which is written more explicitly as 
\begin{equation}
\eta_y^\Delta(\theta) = \eta_y(u_r) + A(u_r, \psi_r;\theta) - F(\psi_r;\theta),
\end{equation}
where $u_r$ and $\psi_r$ are the solutions of the reduced state problem \eqref{eq:state-RB} and adjoint problem \eqref{eq:adjoint-RB}, we form the Lagrangian 
\begin{equation}
\begin{split}
\cL_\eta(u_r, \psi_r, \hat{u}_r, \hat{\psi}_r; \theta) = \eta_y^\Delta(\theta) &+ A(u_r, \hat{u}_r;\theta) - F(\hat{u}_r;\theta) \\
&+ A(\hat{\psi}_r, \psi_r; \theta) + \nabla_u \eta_y|_{u_r}(\hat{\psi}_r), 
\end{split}
\end{equation}
where $\hat{u}_r \in V_r$ and $\hat{\psi}_r \in W_r$ are the Lagrange multipliers. By setting the variation of the Lagrangian w.r.t.\ $\psi_r$ as zero, we have the incremental adjoint problem: find $\hat{\psi}_r \in W_r$ such that
\begin{equation}\label{eq:psihat}
A(\hat{\psi}_r, w_r; \theta) = F(w_r;\theta) - A(u_r, w_r;\theta), \quad \forall w_r \in W_r,
\end{equation} 
whose algebraic system is given by
\begin{equation}
\bA_r^{\psi}(\theta) \hat{\bspsi}_r = \bsf_r^\psi(\theta) - \bA_r^{u, \psi}(\theta) \bsu_r,
\end{equation}
where $\hat{\bspsi}_r$ is the coefficient of $\hat{\psi}_r$ in the basis of $W_r$, $\bA_r^\psi(\theta)$, $\bsf_r^\psi(\theta)$, and $\bA_r^{u, \psi}(\theta)$ are given in \eqref{eq:A-r-psi}, \eqref{eq:frpsi}, and \eqref{eq:Arupsi}, respectively.
By setting the variation of the Lagrangian w.r.t.\ $u_r$ as zero, we have the incremental state problem: find $\hat{u}_r \in V_r$ such that
\begin{equation}
A(v_r, \hat{u}_r;\theta) = - A(v_r, \psi_r;\theta) -\partial_u \eta_y|_{u_r}(v_r)  - \partial_{u}^2 \eta_y|_{u_r}(\hat{\psi}_r, v_r), \quad \forall v_r \in V_r,
\end{equation}
whose algebraic system is given by
\begin{equation}
\bA_r^u(\theta) \hat{\bsu}_r = - (\bA_r^{u, \psi}(\theta))^T \bspsi_r - (y - \bsu_r^T \bO_r^u)^T \Gamma^{-1} \bO_r^u - \hat{\bspsi}_r^T \bO_r^\psi \Gamma^{-1} \bO_r^u,
\end{equation}
where $\hat{\bsu}_r$ is the coefficient vector of $\hat{u}_r$ in the basis of $V_r$. 
Consequently, we obtain the gradient 
\begin{equation}\label{eq:etagraddelta}
\begin{split}
\nabla_\theta \eta_y^\Delta(\theta) 
&= \partial_\theta \cL_\eta(u_r, \psi_r, \hat{u}_r, \hat{\psi}_r; \theta)\\
& = \partial_\theta A(u_r, \psi_r;\theta) - \partial_\theta  F(\psi_r;\theta)  \\
& + \partial_\theta A(u_r, \hat{u}_r;\theta) - \partial_\theta  F(\hat{u}_r;\theta) + \partial_\theta A(\hat{\psi}_r, \psi_r;\theta)\\
& = \bspsi^T_r \partial_\theta \bA_r^{u, \psi}(\theta) \bsu_r - \bspsi_r^T \partial_\theta \bsf_r^\psi(\theta)\\
& + \hat{\bsu}_r^T \partial_\theta \bA_r^{u}(\theta) \bsu_r - \hat{\bsu}_r^T \partial_\theta \bsf_r^u(\theta) + \bspsi_r^T \partial_\theta \bA_r^\psi(\theta) \hat{\bspsi}_r
\end{split}
\end{equation}

\subsection{Affine decomposition}

For the sake of computational reduction enabled by parameter affine decomposition, we assume that
for any $\theta \in \Theta$, the parametric bilinear form $A(\cdot, \cdot;\theta): V \times V \to \bR$ and linear form $F(\cdot:\theta): V \to \bR$ admit the $\theta$-affine decomposition 
\begin{equation}\label{eq:affine}
A(w, v;\theta) = \sum_{j=1}^{J_A} c_j^A(\theta) A_j(w,v) \text{ and } F(v; \theta) = \sum_{j = 1}^{J_F} c_j^F(\theta) F_j(v),
\end{equation}
for some $J_A, J_F \in \bN$, continuously differentiable functions $c_j^A(\theta), c_j^F(\theta)$ w.r.t.\ $\theta$, and $\theta$-independent bilinear form $A_j$ and linear form $F_j$, 
We remark that if the affine decomposition \eqref{eq:affine} is not satisfied, we can apply additional approximation or hyper reduction, e.g., by empirical interpolation \cite{BarraultMadayNguyenEtAl04, ChaturantabutSorensen10}, to achieve the affine decomposition. 

Under the the affine decomposition \eqref{eq:affine}, we have 
\begin{equation}
\bA_r^u(\theta) = \sum_{j = 1}^{J_A} c_j^A(\theta) \bA_{j}^u, \quad \bA_r^\psi(\theta) = \sum_{j = 1}^{J_A} c_j^A(\theta) \bA_{j}^\psi, \quad \bA_r^{u, \psi}(\theta) = \sum_{j = 1}^{J_A} c_j^A(\theta) \bA_{j}^{u,\psi}, 
\end{equation}
where 
\begin{equation}
(\bA_j^u)_{m,n} = A_j(\phi_u^n, \phi_u^m), \quad (\bA_j^\psi)_{m,n} = A_j(\phi_\psi^n, \phi_\psi^m), \quad (\bA_j^{u,\psi})_{m,n} = A_j(\phi_\psi^n, \phi_u^m)
\end{equation}
are computed only once, and $\bA_r^u(\theta)$, $\bA_r^\psi(\theta)$, $\bA_r^{u, \psi}(\theta)$ are assembled for each $\theta \in \Theta$ with $J_A (N_r^u)^2$, $J_A (N_r^\psi)^2$, $J_A N_r^u N_r^\psi$ operations, respectively, which are independent of the high-fidelity degrees of freedom $N_h$. Similarly, we have the decomposition
\begin{equation}
\bsf_r^u(\theta) = \sum_{j=1}^{J_F} c_j^F(\theta) \bsf^u_j, \quad \bsf_r^\psi(\theta) = \sum_{j=1}^{J_F} c_j^F(\theta) \bsf^\psi_j,
\end{equation}
where 
\begin{equation}
(\bsf^u_j)_m = F_j(\phi_u^m), \quad (\bsf^\psi_j)_m = F_j(\phi_\psi^m),
\end{equation}
are computed only once, and $\bsf_r^u(\theta)$ and $\bsf_r^\psi(\theta)$ are assembled for each $\theta \in \Theta$ with $J_F N_r^u$ and $J_F N_r^\psi$ operations, respectively, which are independent of $N_h$.
Similarly, all the gradients w.r.t.\ the parameter $\theta$ can be efficiently decomposed.

\subsection{Adaptive greedy algorithm}

To construct the basis functions of the reduced basis space $V_r$ and $W_r$, which are used to compute the gradient of the log-posterior $\nabla_\theta \log(p_y(\theta_m^l))$ in the SVGD Algorithm \ref{alg:SVGD} for all $\theta_m^l$, $m = 1, \dots, M$, $l = 1, 2, \dots$,
we propose an adaptive greedy algorithm along the progressive construction of the transport map by SVGD, which consists of the key elements: 
\begin{enumerate}
    \item At the initial step, we use the initial samples $\theta_m^0$, $m=1, \dots, M$ in Algorithm \ref{alg:SVGD}, drawn from the prior distribution, as the training samples, and use a classical greedy algorithm with the error indicator $\triangle_r^\eta$ in \eqref{eq:Delta} and tolerance $\varepsilon_r^0$ to construct the reduced order models with the reduced basis spaces $V_r$ and $W_r$.
    \item Then, for a given update criterion being satisfied at step $l = 0, l_1, l_2, \dots, $ e.g., $l_i = i K$ for some $K \in \bN$, we take the samples $\theta_m^l$, $m = 1, \dots, M$, as the training samples and run the greedy algorithm with tolerance $\varepsilon_r^l$ to adaptively enrich the reduced order models and the reduced basis spaces $V_r$ and $W_r$. 
    \item As empirical distribution of the samples $\theta_m^l$, $m = 1, \dots, M$, approaches the true posterior with increasing $l$, we propose to decrease the tolerance $\varepsilon_r^l$ such that the RB approximations are inexpensive to obtain for small $l$ when the samples are far from the posterior, and become more accurate when the samples become closely distributed as the posterior. The tolerance $\varepsilon_r^l$ can be decreased, e.g., according to error indicator $t_l$ in \eqref{eq:t-l}, e.g., $\varepsilon_r^l = \varepsilon_r^l t_l/t_0$.
\end{enumerate}
This adaptive construction process is summarized in Algorithm \ref{alg:adaptive-greedy}.  


\begin{algorithm} 
\caption{Adaptive greedy algorithm} 
\label{alg:adaptive-greedy} 
\begin{algorithmic}[1] 
\STATE \textbf{Input:} random samples $\theta_m^0 \sim \mu_0$, $m = 1, \dots, M$, tolerance $\varepsilon_r^0$, update step $k$.
\STATE \textbf{Output:} Stein samples $\theta_m$, $m = 1, \dots, M$.
\STATE Initialization: at $\theta = \theta_1^0$, solve the high-fidelity problems \eqref{eq:state-HiFi} and \eqref{eq:adjoint-HiFi} for $u_h$ and $\psi_h$, set $V_r = \text{span}\{\phi_u^1\}$ with $\phi_u^1 = u_h/||u_h||_V$ and $W_r = \text{span}\{\phi_\psi^1\}$ with $\phi_\psi^1 = \psi_h/||\psi_h||_V$, compute the reduced matrices and vectors in Section \ref{sec:RB}.
    \WHILE{at step $l = 0, l_1, l_2, \dots, $ between line 4 and 5 in Algorithm \ref{alg:SVGD}}
        \STATE Compute the error indicator $\triangle_r^\eta(\theta_m^l)$ in \eqref{eq:Delta} for $m = 1, \dots, M$.
        \WHILE{$\max_{m = 1, \dots, M}|\triangle_r^\eta(\theta_m^l)| > \varepsilon_r$}
            \STATE Choose $\theta = \argmax_{\theta_m^l, m =1, \dots, M} |\triangle_r^\eta(\theta_m^l)|$.
            \STATE Solve the high-fidelity problems \eqref{eq:state-HiFi} and \eqref{eq:adjoint-HiFi} for $u_h$ and $\psi_h$ at $\theta$.
            \STATE Enrich the spaces $V_r = V_r \bigoplus \text{span}\{u_h\}$, $W_r = W_r \bigoplus \text{span}\{\psi_h\}$.
            \STATE Compute all the reduced matrices and vectors in Section \ref{sec:RB}.
            \STATE Compute the error indicator $\triangle_r^\eta(\theta_m^l)$ in \eqref{eq:Delta} for $m = 1, \dots, M$.
        \ENDWHILE
        \STATE Perform line 5 -- 8 of the SVGD Algorithm \ref{alg:SVGD} with RB approximations.
        \STATE Update the tolerance $\varepsilon_r^l$ according to $t_l$ in Algorithm \ref{alg:SVGD}.
    \ENDWHILE
\end{algorithmic}
\end{algorithm}


\section{Error estimates}
\label{sec:errorestimates}

In this section, we present estimates for the errors between the high-fidelity approximations and the reduced basis approximations of 
the potential $\eta_y(u(\theta))$ and its gradient $\nabla_\theta \eta_y(u(\theta))$, as well as the posterior distribution $\mu_y(\theta)$. We leave the estimate for the errors committed to the samples in Appendix \ref{sec:estimates4samples}.

\subsection{Well-posedness and stability estimates}
We make the following assumptions for the parametric problem \eqref{eq:state}.
\begin{assumption}\label{ass:AF} We assume that the parametric bilinear form $A(\cdot, \cdot;\theta): V \times V \to \bR$ and linear form $F(\cdot;\theta):V \times \bR \to \bR$ satisfy 
\begin{itemize}
    \item[A1] At any $\theta \in \Theta$, there exist a coercivity constant $\alpha(\theta) > 0$ and a continuity constant $\gamma(\theta) > 0$ such that 
    \begin{equation}
    \alpha(\theta) ||w||_V^2 \leq A(w, w; \theta) \text{ and } A(w, v; \theta) \leq \gamma(\theta)||w||_V ||v||_V, \; \forall w, v \in V.
    \end{equation}
    The linear functional $F(\cdot; \theta): V \to \bR$ is bounded with norm 
    \begin{equation}
    ||F(\cdot; \theta)||_{V'} < \infty.
    \end{equation}
    \item[A2] Moreover, $A(\cdot, \cdot; \theta)$ and $F(\cdot; \theta)$ are continuously differentiable w.r.t.\ $\theta$ at every $\theta \in \Theta$, and for each $j = 1, \dots, d$, there exists $0 < \rho_j(\theta) < \infty $ such that
    \begin{equation}
    \partial_{\theta_j} A(w, v; \theta) \leq \rho_j(\theta) ||w||_V ||v||_V, \; \forall w, v \in V, \text{ and }  ||\partial_{\theta_j} F(\cdot; \theta)||_{V'} < \infty.
    \end{equation}
\end{itemize}
\end{assumption}

\begin{lemma}
Under Assumption \ref{ass:AF}, for any $\theta \in \Theta$ there exists a unique solution $u(\theta) \in V$ which satisfies the stability estimate
\begin{equation}\label{eq:stability-u}
||u(\theta)||_V \leq \frac{||F(\cdot;\theta)||_{V'}}{\alpha(\theta)} = : C_u(\theta).
\end{equation}
Moreover, there exists a unique solution $\psi(\theta) \in V$ of the adjoint problem \eqref{eq:adjoint} for each $\theta \in \Theta$, which satisfies the stability estimate
\begin{equation}\label{eq:stability-psi}
||\psi(\theta)||_V \leq \frac{C_y}{\alpha(\theta)} + \frac{C_\cO}{\alpha(\theta)} C_u(\theta) =: C_\psi(\theta),
\end{equation}
where the constants $C_y$ and $C_\cO$ are defined as 
\begin{equation}\label{eq:Gamma-cO-y}
C_y := ||\Gamma^{-1}||_2 ||\cO||_{V'}||y||_2 \quad \text{ and } C_\cO := ||\Gamma^{-1}||_2 ||\cO||_{V'}^2,
\end{equation}
where 
$||\Gamma^{-1}||_2$ is the spectral norm of $\Gamma^{-1}$, $||y||_2$ is the Eculidean norm of $y$, and 
$||\cO||_{V'} := \left(||o_1||_{V'}^2 + \cdots ||o_s||_{V'}^2\right)^{1/2}$ for the observation functional $o_1, \dots, o_s \in V'$.

Furthermore, \eqref{eq:stability-u} holds for the state solutions $u_h$ of \eqref{eq:state-HiFi} and $u_r$ of \eqref{eq:state-RB}, while \eqref{eq:stability-psi} holds for the adjoint solutions $\psi_h$ of \eqref{eq:adjoint-HiFi} and $\psi_r$ of \eqref{eq:adjoint-RB}. 
\end{lemma}

\begin{proof}
The results of well-posedness and estimates \eqref{eq:stability-u} and \eqref{eq:stability-psi} are obtained by a direct application of Lax--Migram theorem. By construction $V_r \subset V_h$, $W_r \subset V_h $, and $V_h \subset V$, which make A1 Assumption \ref{ass:AF} hold in $V_r$, $W_r$, and $V_h$, therefore the same stability estimate \eqref{eq:stability-u} holds for $u_h$ and $u_r$, and \eqref{eq:stability-psi} holds for $\psi_h$ and $\psi_r$.
\end{proof}


\begin{lemma}\label{lemma:nabla-u-phi}
Under Assumption \ref{ass:AF}, for any $\theta \in \Theta$ we have $\nabla_\theta u(\theta) \in V^d := \bigotimes_{i=1}^d V$ with $||\nabla_\theta u(\theta) ||_{V^d} := \sum_{j=1}^d ||\partial_{\theta_j} u(\theta)||_V$, and there holds
\begin{equation}\label{eq:partial-theta-u}
||\nabla_\theta u(\theta) ||_{V^d} \leq \frac{C_u(\theta)}{\alpha(\theta)} \sum_{j = 1}^d \rho_j(\theta) + \sum_{j = 1}^d \frac{||\partial_{\theta_j}F(\cdot;\theta)||_{V'}}{\alpha(\theta)},
\end{equation}
with $C_u(\theta)$ defined in \eqref{eq:stability-u}. Moreover, we have $\nabla_\theta \psi(\theta) \in V^d$, which satisfies 
\begin{equation}\label{eq:partial-theta-psi}
||\nabla_\theta \psi(\theta) ||_{V^d} \leq \frac{C_\psi(\theta)}{\alpha(\theta)} \sum_{j = 1}^d \rho_j(\theta) + \frac{d C_y}{\alpha(\theta)}  + \frac{C_\cO}{\alpha(\theta)} ||\nabla_\theta u(\theta)||_{V^d},
\end{equation}
where  $C_\psi(\theta)$ is defined in \eqref{eq:stability-psi}, $C_y$ and $C_\cO$ are defined in \eqref{eq:Gamma-cO-y}.
The same estimate \eqref{eq:partial-theta-u} holds for $\nabla_\theta u_h(\theta)$ and $\nabla_\theta u_r(\theta)$, and \eqref{eq:partial-theta-psi} holds for $\nabla_\theta \psi_h(\theta) $ and $\nabla_\theta \psi_r(\theta)$.
\end{lemma}
The proof of this lemma is in Appendix \ref{sec:proof-lemma-1}.

\subsection{Error estimates for the potential $\eta_y(u)$ and its gradient $\nabla_\theta \eta_y(u)$}
For notational convenience, let $e_r^u(\theta)$ and $e_r^\psi(\theta)$ denote the reduced basis approximation errors of the state and adjoint, i.e., 
\begin{equation}
e_r^u(\theta): = u_h(\theta) - u_r(\theta), \quad \text{ and } e_r^\psi(\theta) : = \psi_h(\theta) - \psi_r(\theta),
\end{equation}
for which and their gradients $\nabla_\theta e_r^u$ and $\nabla_\theta e_r^\psi$, the a-posteriori error estimates are presented in Appendix \ref{sec:upsigrad}.
We denote the high-fidelity and reduced basis approximations of $\eta_y(u(\theta))$ as 
\begin{equation}
\eta^h_y(\theta) = \eta_y(u_h(\theta)), \text{ and } \eta_y^r(\theta) = \eta_y(u_r(\theta)),
\end{equation}
and denote the reduced basis approximation errors as 
\begin{equation}
e_r^\eta(\theta) := \eta_y^h(\theta) - \eta_y^r(\theta), \quad \text{ and } e_r^\Delta(\theta) = \eta_y^h(\theta) - \eta_y^\Delta(\theta), 
\end{equation}
where $\eta_y^\Delta$ is defined in \eqref{eq:dwreta}. 

\begin{lemma}\label{lemma:bound-eta-h-r}
Under Assumption \ref{ass:AF}, for any $\theta \in \Theta$, we have
\begin{equation}\label{eq:etabound}
|e_r^\eta(\theta)| \leq (C_y + C_\cO C_u(\theta)) ||e_r^u(\theta)||_V,
\end{equation}
and 
\begin{equation}\label{eq:e-eta-Delta-eta}
|e_r^\Delta(\theta)| \leq  \gamma(\theta) ||e_r^u(\theta)||_V ||e_r^\psi(\theta)||_V + \frac{1}{2} C_\cO ||e_r^u(\theta)||_V^2,
\end{equation}
where the constants $C_y$ and $C_\cO$ are defined in \eqref{eq:Gamma-cO-y}, $C_u(\theta)$ is defined in \eqref{eq:stability-u}. 

\end{lemma}
\begin{proof}
By definition of the high-fidelity approximation and reduced basis approximation of the potential $\eta_y$ in \eqref{eq:eta-HiFi} and \eqref{eq:eta-RB}, we have 
\begin{equation}
\begin{split}
& \eta_y(u_h) - \eta_y(u_r)  \\
& = \frac{1}{2} (y - \cO(u_h))^T \Gamma^{-1} (y - \cO(u_h))  - \frac{1}{2} (y - \cO(u_h))^T \Gamma^{-1} (y - \cO(u_r)) \\
& + \frac{1}{2} (y - \cO(u_h))^T \Gamma^{-1} (y - \cO(u_r)) - \frac{1}{2} (y - \cO(u_r))^T \Gamma^{-1} (y - \cO(u_r)),
\end{split}
\end{equation}
where the first two terms can be bounded by 
\begin{equation}
\frac{1}{2} ||\Gamma^{-1}||_2 ||\cO||_{V'} ||e_r^u||_V (||y||_2 + ||\cO||_{V'} ||u_h||_V),
\end{equation}
while the last two terms can be bounded by 
\begin{equation}
\frac{1}{2} ||\Gamma^{-1}||_2 ||\cO||_{V'} ||e_r^u||_V (||y||_2 + ||\cO||_{V'} ||u_r||_V),
\end{equation}
which, together with the stability estimates \eqref{eq:stability-u} for $u_h$ and $u_r$, concludes.



By Taylor expansion of $\eta_y(u_r)$ at $u_h$, which is quadratic w.r.t.\ $u_h$, we have 
\begin{equation}
\eta_y(u_r) - \eta_y(u_h) = - \nabla_u \eta_y|_{u_h} (e_r^u) + \frac{1}{2} \nabla_u^2 \eta_y|_{u_h}(e_r^u, e_r^u),
\end{equation}
For the first term, by the adjoint high-fidelity problem \eqref{eq:adjoint-HiFi}, and $e_r^u \in V_h$, we have 
\begin{equation}
-\nabla_u \eta_y|_{u_h} (e_r^u) = A(e_r^u, \psi_h; \theta).
\end{equation}
Moreover, by the definition of $\Delta_r^\eta$ in \eqref{eq:Delta} and the state high-fidelity problem \eqref{eq:state-HiFi} with $\psi_r \in W_r \subset V_h$, we have 
\begin{equation}
\Delta_r^\eta(\theta) = A(u_r, \psi_r; \theta) - A(u_h, \psi_r; \theta) = -A(e_r^u, \psi_r; \theta). 
\end{equation}
Therefore, by the definition $e_r^\Delta(\theta) = \eta_y(u_h) - \eta_y(u_r) - \Delta_r^\eta(\theta)$ we have
\begin{equation}\label{eq:e-delta-r}
e_r^\Delta(\theta) =  - A(e_r^u, e_r^\psi; \theta) -  \frac{1}{2} \nabla_u^2 \eta_y|_{u_h}(e_r^u, e_r^u)
\end{equation}
which concludes under Assumption \ref{ass:AF}.


\end{proof}



\begin{lemma}\label{lemma:grad-e-r-eta}
Under Assumption \ref{ass:AF}, for any $\theta \in \Theta$, we have 
\begin{equation}\label{eq:grad-e-eta}
 ||\nabla_\theta e_r^\eta(\theta)||_1  \leq (C_y+C_\cO C_u(\theta))  ||\nabla_{\theta} e_r^u(\theta)||_{V^d} + C_\cO ||\nabla_\theta u_r(\theta)||_{V^d} ||e_r^u(\theta)||_V,
\end{equation}
and 
\begin{equation}\label{eq:grad-e-Delta}
\begin{split}
 ||\nabla_\theta e_r^\Delta(\theta)||_1 & \leq \gamma(\theta) ||\nabla_\theta e_r^u(\theta)||_{V^d} ||e_r^\psi(\theta)||_V  + \gamma(\theta) ||\nabla_\theta e_r^\psi(\theta)||_{V^d} ||e_r^u(\theta)||_V\\
& + \sum_{j=1}^d\rho_j(\theta)  ||e_r^u(\theta)||_V  ||e_r^\psi(\theta)||_V + C_\cO ||e_r^u(\theta)||_V ||\nabla_\theta e_r^u(\theta)||_{V^d}.
\end{split}
\end{equation}
where $||\cdot||_1$ denotes the $\ell_1$-norm, i.e., $||g||_1 = \sum_{j=1}^d |g_j|$ for $g = (g_1, \dots, g_d)\in \bR^d$, and the constants $C_y$ and $C_\cO$ are defined in \eqref{eq:Gamma-cO-y}, $C_u(\theta)$ is defined in \eqref{eq:stability-u}. 
\end{lemma}
\begin{proof}
For any $j = 1, \dots, d$, by definition of $\eta_y(\theta)$ and Lemma \ref{lemma:nabla-u-phi}, we have 
\begin{equation}
\begin{split}
& \nabla_{\theta_j} \eta_y^h(\theta) - \nabla_{\theta_j} \eta_y^r(\theta) \\
& = -(y - \cO(u_h(\theta)))^T \Gamma^{-1} \cO(\nabla_{\theta_j} u_h(\theta)) + (y - \cO(u_r(\theta)))^T \Gamma^{-1} \cO(\nabla_{\theta_j} u_r(\theta))\\
& = -(y - \cO(u_h(\theta)))^T \Gamma^{-1} \cO(\nabla_{\theta_j} u_h(\theta)) + (y - \cO(u_h(\theta)))^T \Gamma^{-1} \cO(\nabla_{\theta_j} u_r(\theta))\\
& \quad -(y - \cO(u_h(\theta)))^T \Gamma^{-1} \cO(\nabla_{\theta_j} u_r(\theta)) + (y - \cO(u_r(\theta)))^T \Gamma^{-1} \cO(\nabla_{\theta_j} u_r(\theta)),
\end{split}
\end{equation}
which implies 
\begin{equation}
\begin{split}
|\nabla_{\theta_j} e_r^\eta(\theta)| \leq (C_y + C_\cO ||u_h(\theta)||_V) ||\partial_{\theta_j} e_r^u(\theta)||_V + C_\cO ||\partial_{\theta_j} u_r(\theta)||_V ||e_r^u(\theta)||_V,
\end{split}
\end{equation}
which concludes \eqref{eq:grad-e-eta} by the stability estimate \eqref{eq:stability-u} and summing over $j = 1, \dots, d$.
By definition of $\eta^\Delta_y(\theta)$ and the relation \eqref{eq:e-delta-r}, for any $j = 1, \dots, d$, by the continuous differentiability of $A(\cdot, \cdot; \theta)$ w.r.t.\ $\theta$, and Lemma \ref{lemma:nabla-u-phi}, we have 
\begin{equation}
\begin{split}
&\partial_{\theta_j} \eta_y^h(\theta) - \partial_{\theta_j} \eta^\Delta_y(\theta) 
 \\
 & = - A(\partial_{\theta_j} e_r^u(\theta), e_r^\psi(\theta); \theta)  - A( e_r^u(\theta), \partial_{\theta_j} e_r^\psi(\theta); \theta) \\
& - \partial_{\theta_j} A( e_r^u(\theta),  e_r^\psi(\theta); \theta) -  \nabla_u^2 \eta_y|_{u_h}(\partial_{\theta_j} e_r^u(\theta), e_r^u(\theta)),
\end{split}
\end{equation}
which, under Assumption \ref{ass:AF}, can be bounded by 
\begin{equation}
\begin{split}
& |\partial_{\theta_j} \eta_y^h(\theta) - \partial_{\theta_j} \eta^\Delta_y(\theta) | \\
& \leq \gamma(\theta) ||\partial_{\theta_j} e_r^u(\theta)||_V ||e_r^\psi(\theta)||_V + \gamma(\theta) ||e_r^u(\theta)||_V ||\partial_{\theta_j} e_r^\psi(\theta)||_V \\
& + \rho_j(\theta) ||e_r^u(\theta)||_V ||e_r^\psi(\theta)||_V + C_\cO ||e_r^u(\theta)||_V ||\partial_{\theta_j} e_r^u(\theta)||_V,
\end{split}
\end{equation}
which concludes \eqref{eq:grad-e-Delta} by summing over $j = 1, \dots, d$.
\end{proof}

\subsection{Error estimates for the posterior $\mu_y$}
Let $\mu_y^h$, $\mu_y^r$, and $\mu_y^\Delta$ denote the posterior distributions with densities $p_y^h(\theta)$, $p_y^r(\theta)$, and $p_y^\Delta(\theta)$ at $\theta \in \Theta$ by the high-fidelity and reduced basis approximations of the potential $\eta_y^h(\theta)$, $\eta_y^r(\theta)$, $\eta_y^\Delta(\theta)$, respectively. More explicitly, we write 
\begin{equation}
p_y^h(\theta) = \frac{1}{Z_h} \exp(-\eta_y^h(\theta)) p_0(\theta), 
\end{equation}
where the normalization constant $Z_h$ is given by
\begin{equation}
Z_h = \int_\Theta \exp(-\eta_y^h(\theta)) p_0(\theta) d\theta.
\end{equation}
The densities $p_y^r$ and $p_y^\Delta$, and the normalization constants $Z_r$ and $Z_\Delta$ are defined similarly corresponding to the potential $\eta_y^r$ and $\eta_y^\Delta$, respectively.

\begin{theorem}\label{thm:DKL-distance}
Under Assumption \ref{ass:AF}, we have 
\begin{equation}\label{eq:KL-r}
D_{\text{KL}} (\mu_y^h | \mu_y^r) \leq \bE_{\mu_y^h}\left[
|e_r^\eta|\right] + \bE_{\mu_y^h}[|\exp(e_r^\eta) - 1|],
\end{equation}
and 
\begin{equation}\label{eq:KL-Delta}
D_{\text{KL}} (\mu_y^h | \mu_y^\Delta) \leq \bE_{\mu_y^h}\left[
|e_r^\Delta|\right] + \bE_{\mu_y^h}[|\exp(e_r^\Delta) - 1|].
\end{equation}
\end{theorem}

\begin{proof}
By definition of the KL divergence in \eqref{eq:KL-divergence}, we have 
\begin{equation}
\begin{split}
D_{\text{KL}} (\mu_y^h | \mu_y^r) & = \int_{\theta \in \Theta} p_y^h(\theta) \log \left(
\frac{p_y^h(\theta)}{p_y^r(\theta)}
\right) d\theta\\
& = \int_\Theta p_y^h(\theta) (\eta_y^r(\theta) - \eta_y^h(\theta)) d\theta + \log\left(\frac{Z_r}{Z_h}\right),  
\end{split}
\end{equation}
where the first term can be bounded by 
\begin{equation}
\begin{split}
\int_\Theta p_y^h(\theta) (\eta_y^r(\theta) - \eta_y^h(\theta)) d\theta &\leq \int_\Theta p_y^h(\theta) |e_r^\eta(\theta)| d\theta.
\end{split}
\end{equation}
To bound the second term $\log(Z_r/Z_h)$, we have 
\begin{equation}
\begin{split}
|Z_r - Z_h| & = \left|\int_\Theta (\exp(-\eta_y^r(\theta)) - \exp(-\eta_y^h(\theta))) p_0(\theta) d\theta \right|\\
& \leq \int_\Theta |\exp(-\eta_y^r(\theta)) - \exp(-\eta_y^h(\theta))| p_0(\theta) d\theta\\
& \leq \int_\Theta |\exp(e_r^\eta(\theta)) - 1| \exp(-\eta_y^h(\theta)) p_0(\theta) d\theta \\
& = Z_h \int_\Theta |\exp(e_r^\eta(\theta)) - 1| p_y^h(\theta) d\theta.
\end{split}
\end{equation}
Moreover, we have by $\log(1+\tau) \leq \tau $ for $\tau \geq 0$ that 
\begin{equation}
\begin{split}
\log\left(\frac{Z_r}{Z_h}\right)  \leq \log\left(1 + \frac{|Z_r - Z_h|}{Z_h}\right)  \leq \frac{|Z_r - Z_h|}{Z_h}.
\end{split}
\end{equation}
A combination of the above estimates concludes \eqref{eq:KL-r}. By the same argument, we obtain \eqref{eq:KL-Delta} where $\eta_y^r$ is replaced by $\eta_y^\Delta$. 
\end{proof}

\begin{remark}
Note that $|e^\tau - 1| < 2\tau$ for $\tau < 1$, so that if $|e_r^\eta(\theta)| < 1$ for all $\theta \in \Theta$, we have 
\begin{equation}
D_{\text{KL}} (\mu_y^h | \mu_y^r) \leq 3 \bE_{\mu_y^h}\left[
|e_r^\eta|\right].
\end{equation}
By the adaptive greedy construction in Algorithm \ref{alg:adaptive-greedy}, $|e_r^\eta(\theta)|$ is not necessarily small (in particular smaller than one) in the whole parameter domain $\Theta$. However, by construction it is small in the region where the posterior density is big. Let $\Omega_1 =:\{\theta \in \Theta: e_r^\eta(\theta) < 1\}$, then as long as $\bE_{\mu_y^h(\Omega \setminus \Omega_1)}[|\exp(e_r^\eta)-1|]$ is small, e.g., $\bE_{\mu_y^h(\Omega \setminus \Omega_1)}[|\exp(e_r^\eta)-1|] < K \bE_{\mu_y^h}\left[
|e_r^\eta|\right]$ for some constant $K > 0$, we have 
\begin{equation}
D_{\text{KL}} (\mu_y^h | \mu_y^r) \leq (3+K) \bE_{\mu_y^h}\left[
|e_r^\eta|\right].
\end{equation}
The same holds for $D_{\text{KL}} (\mu_y^h | \mu_y^\Delta) $. Note that $|e_r^\Delta|$ is typically smaller than $|e_r^\eta|$.
\end{remark}

By Lemma \ref{lemma:bound-eta-h-r} for the bound of the errors $e_r^\eta$  and $e_r^\Delta$, and Lemma  \ref{lemma:bound-u-psi-h-r} for the bound of the errors $e_r^u$ and $e_r^\psi$, Theorem \ref{thm:DKL-distance} implies the following results.

\begin{corollary}
Under Assumption \ref{ass:AF}, we have 
\begin{equation}
D_{\text{KL}} (\mu_y^h | \mu_y^r) \leq \bE_{\mu_y^h}\left[
C_{\alpha}^u ||R_u(u_r, \cdot; \cdot)||_{V'}\right] + \bE_{\mu_y^h}\left[
|\exp(C_{\alpha}^u ||R_u(u_r, \cdot; \cdot)||_{V'}) - 1| \right],
\end{equation}
where the constant $C_\alpha^u(\theta)$ is given by
\begin{equation}
C_\alpha^u(\theta) = \frac{(C_y+C_\cO C_u(\theta))}{\alpha(\theta)},
\end{equation}
where the residual $R_u$ is defined in \eqref{eq:state-residual},  the constants $C_y$ and $C_\cO$ are defined in \eqref{eq:Gamma-cO-y}, $C_u(\theta)$ is defined in \eqref{eq:stability-u}, and
\begin{equation}
\begin{split}
D_{\text{KL}} (\mu_y^h | \mu_y^\Delta) \leq \bE_{\mu_y^h}\left[
R_u^\psi(u_r, \psi_r; \cdot) \right] + \bE_{\mu_y^h}\left[
|\exp(R_u^\psi(u_r, \psi_r; \cdot))-1|\right],
\end{split}
\end{equation}
where 
\begin{equation}
R_u^\psi(u_r, \psi_r; \theta) := C_{\alpha}^\gamma ||R_u(u_r, \cdot; \theta)||_{V'} ||R_\psi(\cdot,\psi_r; \theta)||_{V'}+C_\alpha^{\gamma, \cO}||R_u(u_r, \cdot; \theta)||_{V'}^2,
\end{equation}
where $R_\psi$ is defined in \eqref{eq:adjoint-residual}, the constant $C_\alpha^\gamma(\theta)$ and $C_\alpha^{\gamma,\cO}(\theta)$ are given by 
\begin{equation}
C_\alpha^\gamma(\theta) = \frac{\gamma(\theta)}{\alpha^2(\theta)}, \quad \text{ and } C_\alpha^{\gamma,\cO}(\theta) = \frac{2\gamma(\theta) C_\cO +\alpha(\theta) C_\cO}{2\alpha^3(\theta)}.
\end{equation}
\end{corollary}

\section{Numerical experiments}
\label{sec:numerics}


In this section, we perform numerical experiments based on a linear diffusion problem with random coefficient to demonstrate the computational accuracy and efficiency of the proposed method. More specifically, we consider the parametric diffusion problem
\begin{equation}\label{eq:diffusion}
-\nabla \cdot( a(\theta, x) \nabla u(\theta, x) ) = f(x), \quad x \in D = (0, 1)^2,
\end{equation}
with homogeneous Dirichlet boundary condition on the bottom and top boundaries and homogeneous Neumann boundary condition on the left and right boundaries. The parametric coefficient $a(\theta)$ for each $\theta \in \Theta \in \bR^d$ is given by 
\begin{equation}
a(\theta, x) = a_0(x) + \sum_{j=1}^d c_j(\theta) a_j(x), \quad x \in D,
\end{equation}
which leads to the PDE model \eqref{eq:state} with 
\begin{equation}
A(u,v;\theta) = \int_D a(\theta) \nabla u \cdot \nabla v dx, \text{ and } F(v;\theta) = \int_D f v dx,
\end{equation}
where the affine decomposition \eqref{eq:affine} can be explicitly written with 
\begin{equation}
A_j(u,v) = \int_D a_j \nabla u \cdot \nabla v dx, \text{ and } c_j^A(\theta) = c_j(\theta), \quad j = 1, \dots, d.
\end{equation}

We consider pointwise observation operator $\cO = (o_1, \dots, o_{s})$ with $o_i(u(\theta)) = u(\theta, x_i)$, where $x_i$, $i = 1, \dots, 49$, are uniformly located in the domain $D = (0, 1)^2$. For the observation noise $\xi \sim \cN(0, \Gamma)$, we set the covariance $\Gamma = \text{diag}(\sigma^2, \dots, \sigma^2)$ with $\sigma = 0.01 \times \max_{i=1, \dots, s}(o_i(u(\theta_{\text{ref}})))$ where $\theta_{\text{ref}} = (1, \dots, 1)$ is set as the reference value of the parameter.
For the high-fidelity approximation, we use a finite element method with linear elements in a mesh of uniform triangle of size $129 \times 129$. 

We study two parametrization cases with different prior distributions and different number of parameters, one with $4$ uniformly distributed random parameters and the other with $9$ Gaussian random parameters. We present results for the former in this section, while similar results for the latter are presented in Appendix \ref{sec:Gaussain}.

In this example, we consider the parameter $\theta = (\theta_1, \theta_2, \theta_3, \theta_4)$ with i.i.d.\ uniformly distributed random components $\theta_i \sim U([-\sqrt{3}, \sqrt{3}])$, with zero mean and unit variance, the coefficients $c_j(\theta) = \theta_j$, $j = 1, 2,3, 4$, and the basis functions 
$$a_0(x) = 5, \quad a_j(x) = \cos( j_1 \pi x_1) \cos( j_2 \pi x_2),
$$
with $(j_1, j_2) = (1,1), (1,2), (2,1), (2,2)$ for $ j = 1, 2, 3, 4$. We set $f = 1$ in \eqref{eq:diffusion}.

\begin{figure}[!htb]
	\centering
	\includegraphics[trim=0 0 40 40,clip=True,width=0.48\textwidth]{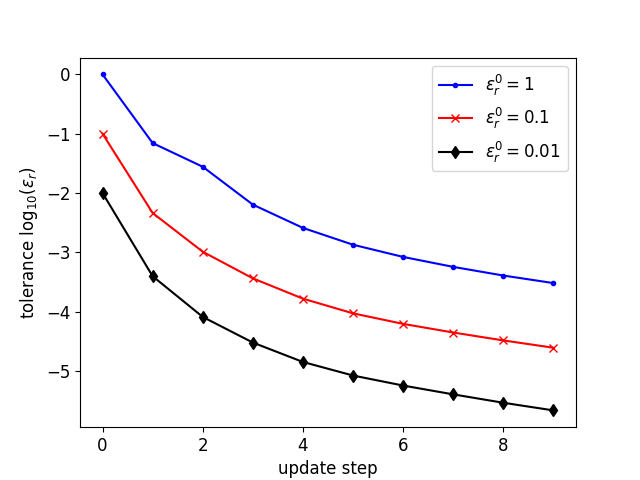}
	\includegraphics[trim=0 0 40 40,clip=True,width=0.48\textwidth]{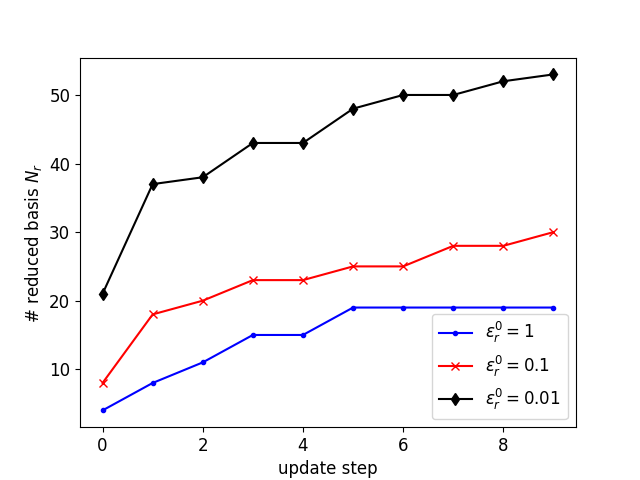}
	\caption{Change of tolerances (left) and the number of reduced basis functions (right) w.r.t.\ the RB update step $i$ with SVGD step $l = i K$ and $K = 10$ in Algorithm \ref{alg:adaptive-greedy}.}
	\label{fig:rom-update}
\end{figure}

We run the greedy Algorithm \ref{alg:adaptive-greedy} for the construction of reduced basis approximations and their applications in the SVGD process. We follow \cite{LiuWang16} to use the kernel  \eqref{eq:kernel} with the scaling factor $h = \text{med}^2/\log(N)$, where $\text{med}$ is the median of the pairwise distance between the current particles $\theta_n$, $n = 1, \dots, N$. We set the initial tolerance as $\varepsilon_r^0 = 1, 0.1, 0.01$, respectively, and update the reduced basis approximations every $K = 10$ SVGD steps with new tolerance given by $\varepsilon_r^l = \varepsilon_r^0 t_l$ with $t_l$ defined in \eqref{eq:t-l}. The change of the tolerances and the number of reduced basis functions for different initial tolerances are shown in Fig.\ \ref{fig:rom-update}, from which we can see that as the gradient norm of SVGD update decreases, i.e., as the particles more closely follow the posterior distribution, the reduced basis approximations become more accurate with a larger number of reduced basis functions. 

\begin{figure}[!htb]
	\centering
	\begin{subfigure}{0.59\textwidth}
		\centering
		\centering
		\includegraphics[trim=80 10 80 30,clip=True,width=0.32\textwidth]{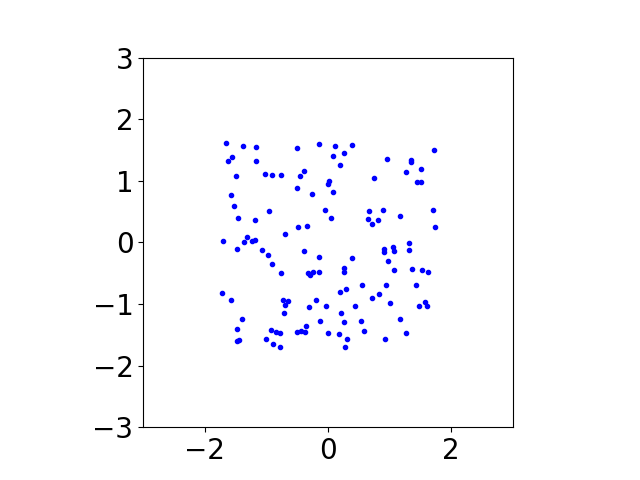}
		\includegraphics[trim=80 10 80 30,clip=True,width=0.32\textwidth]{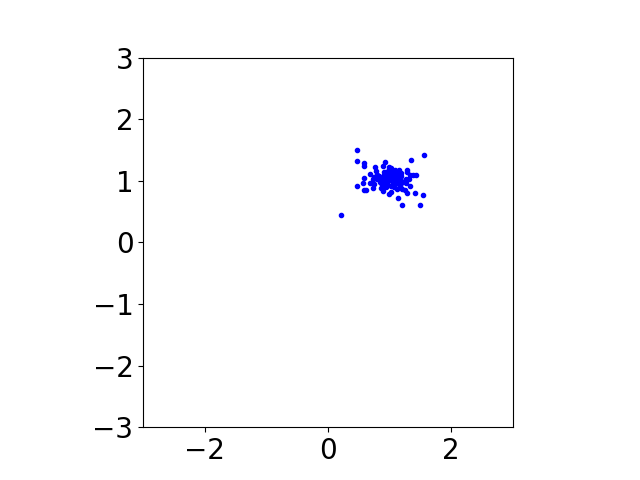}
		\includegraphics[trim=80 10 80 30,clip=True,width=0.32\textwidth]{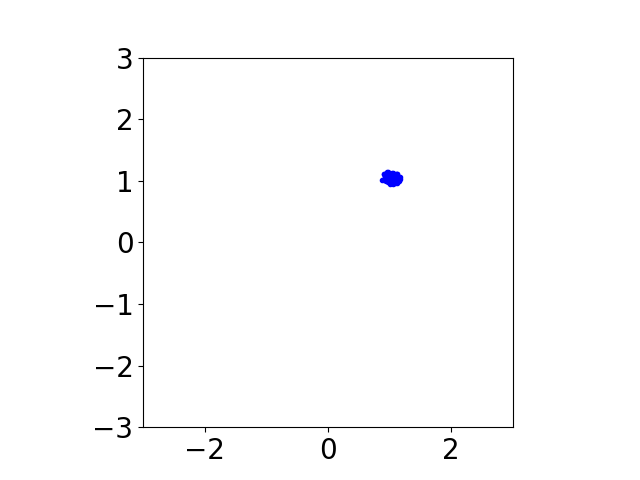}
		
		\includegraphics[trim=80 10 80 30,clip=True,width=0.32\textwidth]{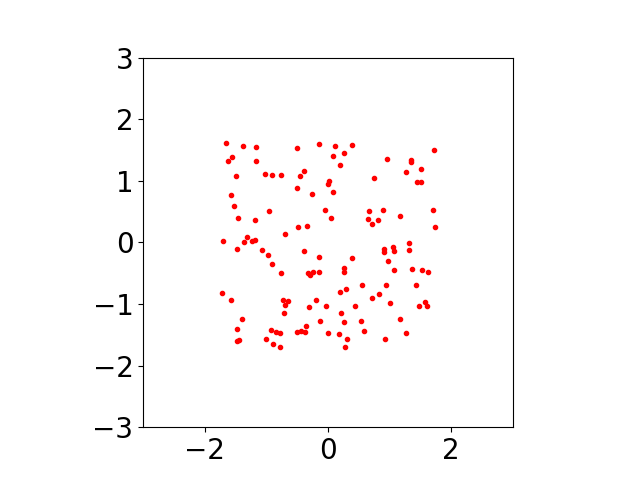}
		\includegraphics[trim=80 10 80 30,clip=True,width=0.32\textwidth]{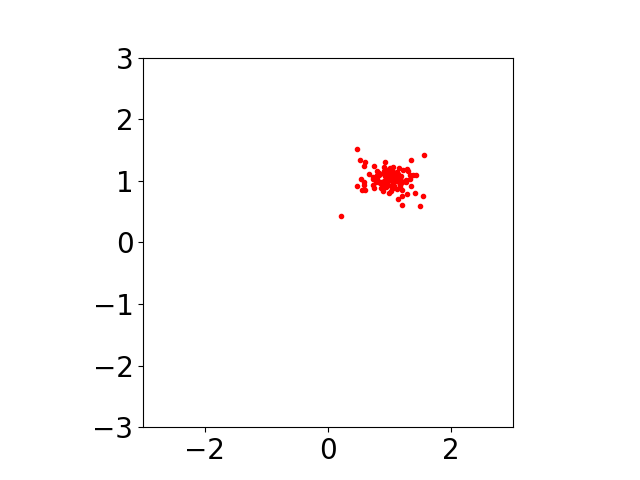}
		\includegraphics[trim=80 10 80 30,clip=True,width=0.32\textwidth]{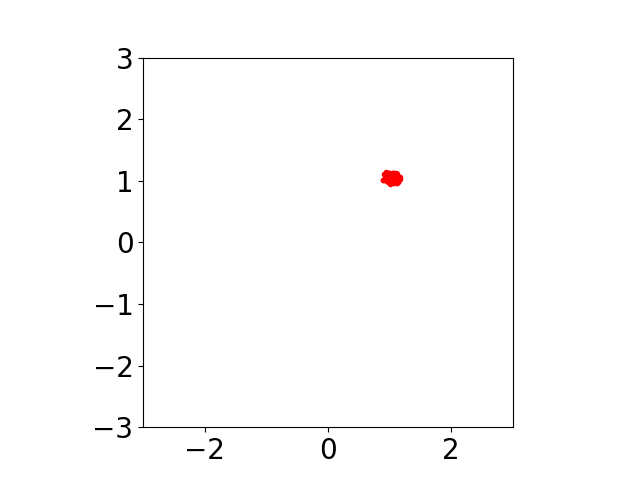}
		\caption{Locations of 128 particles ($\theta_1, \theta_2$) at SVGD step $l = 0$ (left), $9$ (middle), $99$ (right) by high-fidelity (top) and reduced basis (bottom) approximations.}
	\end{subfigure}%
	~ 
	\begin{subfigure}{0.39\textwidth}
		\centering
		\includegraphics[trim=70 10 80 30,clip=True,width=0.96\textwidth]{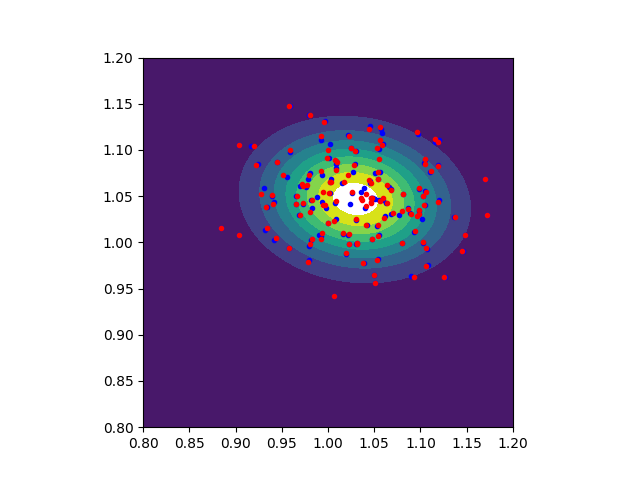}
		\caption{Contour of the marginal posterior density for $(\theta_1, \theta_2)$ and locations of particles at SVGD step $l = 99$.}
	\end{subfigure}
\vspace{-0.5cm}
	\caption{Comparision of particles by high-fidelity and reduced basis approximations.}
	\label{fig:particles-uniform}
\end{figure}

Fig.\ \ref{fig:particles-uniform} depicts the update of 128 particles (projected in dimension ($\theta_1, \theta_2$)) by SVGD with high-fidelity and reduced basis approximations (with initial tolerance $\varepsilon_r^0 = 0.01$) of the PDE models, respectively. At the initial step $l = 0$, we randomly draw 128 samples from the uniform prior distribution, as shown in the left two figures of part (a), and use them for both the high-fidelity and reduced basis approximations. At SVGD step $l = 9$ and $l = 99$, the updated particles with different approximations are displayed in the middle and right two figures of part (a), which appear very close to each other. More details are shown in part (b) in the enlarged region where the marginal posterior density of high-fidelity approximation in dimension ($\theta_1, \theta_2$) is evidently different from zero, from which we can see that the particles obtained by the reduced basis approximations are very close to those obtained by the high-fidelity approximations. Both have effectively good empirical representation of the posterior distribution. 

\begin{figure}[!htb]
    \centering
    \includegraphics[trim=0 0 40 40,clip=True,width=0.48\textwidth]{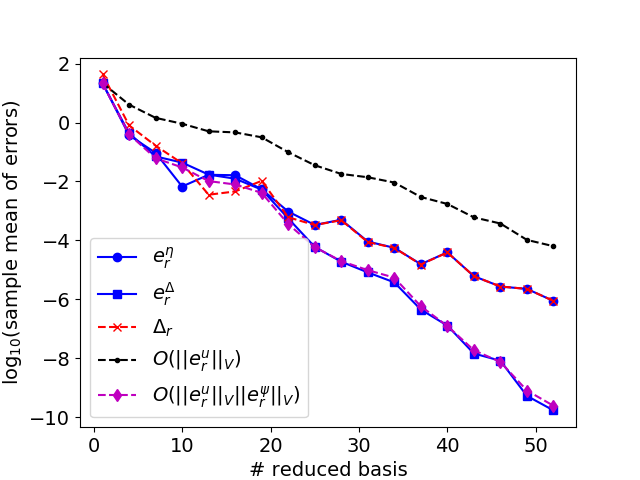}
    \includegraphics[trim=0 0 40 40,clip=True,width=0.48\textwidth]{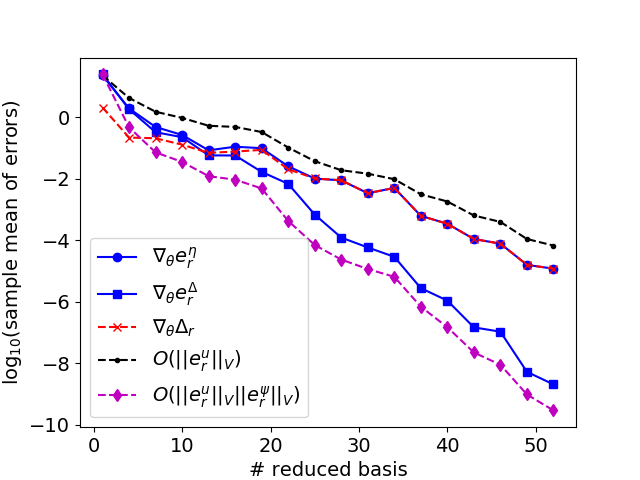}

    \includegraphics[trim=0 0 40 40,clip=True,width=0.48\textwidth]{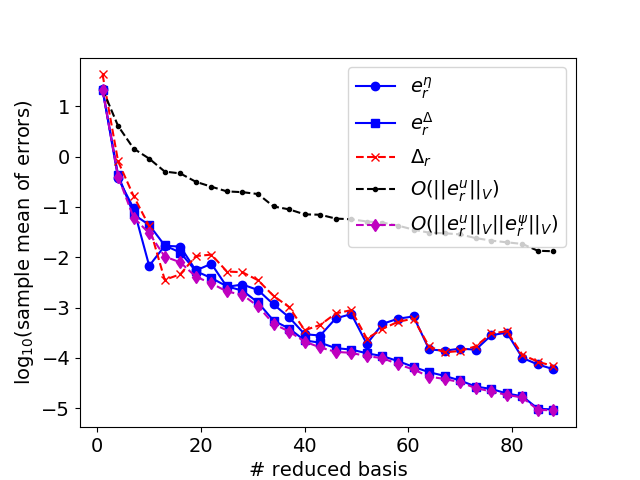}
    \includegraphics[trim=0 0 40 40,clip=True,width=0.48\textwidth]{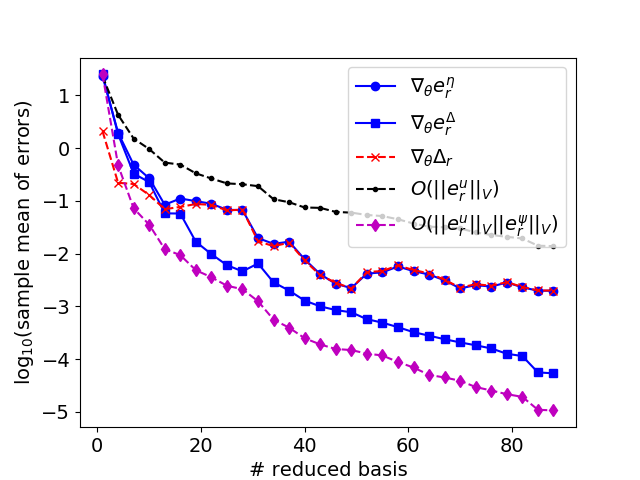}
    \caption{Sample mean of the approximation errors and estimates for adaptive RB (top) and fixed RB (bottom) approximations of $\eta_y^h$ (left) and $\nabla_\theta \eta_y^h$ (right) at step $l = 99$.}
    \label{fig:accuracy-uniform}
\end{figure}

Fig.\ \ref{fig:accuracy-uniform} demonstrates the accuracy of the reduced basis approximations of the potential $\eta_y$ and its gradient $\nabla_\theta \eta_y$, the efficacy of the error estimate $\Delta_r$ used in the greedy algorithm and various error bounds, as well as the advantage of the adaptive greedy construction. More specifically, from the top two figures on the decay of the sample averaged reduced basis approximation errors $e_r^\eta$ and $e_r^\Delta$ at SVGD step $l = 99$, obtained by the adaptive greedy Algorithm \ref{alg:adaptive-greedy}, we can see that the averaged error $e_r^\eta$ decays asymptotically as the averaged error bound $||e_r^u||_V$ (which is rescaled by a constant such that the error and the bound are equal at $N_r = 1$), as predicted by Lemma \ref{lemma:bound-eta-h-r}. Moreover, the averaged error $e_r^\Delta$ decays asymptotically as the averaged error bound $||e_r^u||_V ||e_r^\psi||_V$, as predicted by Lemma \ref{lemma:bound-eta-h-r} where we note that $||e_r^u||_V ||e_r^\psi||_V$ dominates $||e_r^u||_V^2$. By comparison of $e_r^\eta$ and $e_r^\Delta$, we can also see that, $\eta_y^\Delta$, the reduced basis approximation of the potential, $\eta_y^r$, corrected by the dual weighted residual $\Delta_r$, is much more accurate than $\eta_y^r$ itself, especially when the number of reduced basis functions becomes large. This observation can be confirmed by the closeness of the residual $\Delta_r$ and the error $e_r^\eta$ as shown in the top-left part of Fig. \ref{fig:accuracy-uniform}. Similar conclusions consistent with Lemma \ref{lemma:grad-e-r-eta} can be drawn for the reduced basis approximation of the gradient of the potential $\nabla_\theta \eta_y^h$, as depicted in the top right part of the figure. Note that we did not compute the norm of the gradients of the state and adjoint, i.e., $||\nabla_\theta e_r^u||_{V^d}$ and $||\nabla_\theta e_r^\psi||_{V^d}$, as they involve solving additional $2d$ PDE problems presented in Section \ref{sec:upsigrad}, which are not needed in the adaptive greedy algorithm, Algorithm \ref{alg:adaptive-greedy}. The bottom two figures of Fig.\ \ref{fig:accuracy-uniform} show the decay of the errors and bounds for the reduced basis approximation with fixed reduced basis functions constructed at the initial step of SVGD with tolerance $\varepsilon_r = 10^{-5}$, in contrast with the adaptive construction. We can see that the reduced basis approximations constructed once and used for all later SVGD evaluations become less accurate than the reduced basis approximations by the adaptive construction, both for the approximation of $\eta_y$ and $\nabla_\theta \eta_y$, even when the number of reduced basis functions of the former is much larger than the latter. This demonstrates the advantage of the adaptive greedy construction in terms of accuracy of the reduced basis approximations.

\begin{table}[h!]
\centering
\begin{tabular}{| c | c | c | c c c | c|} 
 \hline
\multicolumn{2}{|c|}{} & HiFi & \multicolumn{3}{|c|}{adaptive RB} & fixed RB\\  
 \hline
\multicolumn{2}{|c|}{initial tolerance $\varepsilon_r^0$} & n/a & $1$ &  $0.1$ &  $0.01$ & $0.00001$ \\
\hline
\multirow{4}{*}{$M=64$}&DOF ($N_h, N_r$) & 16641 
& 20 & 31 & 49 & 62\\
& time to build RB & n/a & $4.4$ & $7.1$ & $12.2$ & $15.8$\\
& time for evaluation & $1.8\times 10^3$ & $4.4$ & $4.8$ & $5.8$ & $7.3$ \\
& speedup factor & 1 & 203 & 148 & 98 & 62\\
 \hline
\multirow{4}{*}{M=128}&DOF ($N_h, N_r$) & 16641 
& 19 & 30 & 53 & 87\\
& time to build RB & n/a & $4.5$ & $7.3$ & $14.3$ & $26.3$\\
& time for evaluation & $3.5\times 10^3$ & $8.3$ & $9.5$ & $11.8$ & $19.2$\\
& speedup factor & 1 & 267 & 212 & 137 & 78\\
 \hline
\end{tabular}
\caption{Comparison of computational cost of high-fidelity (HiFi) and reduced basis (RB) approximations for SVGD up to $l = 99$, with different number of particles, different RB construction schemes and tolerances, in terms of degrees of freedom (DOF), CPU time for evaluation and RB construction, and speedup factor, which is the ratio of HiFi evaluation time/(RB construction + evaluation time).}
\label{table:cost}
\end{table}

We report the computational cost of high-fidelity and reduced basis approximations in the SVGD process up to step $l = 99$ in Table \ref{table:cost}. For the high-fidelity approximation, the number of degrees of freedom is 16,641 using P1 finite elements on a $129\times 129$ mesh of triangles, which leads to an averaged (128 samples at SVGD step $l = 99$) approximation error for the potential $\eta_y$ of about $10^{-4}$ (using mesh size $257\times 257$ as reference). From the results we see that with increasing initial tolerance $\varepsilon_r^0 = 1, 0.1, 0.01$, the adaptive RB becomes more expensive for construction and evaluation, which achieves mostly over 100X speedup compared to the high-fidelity solution (HiFi) in terms of CPU time. Moreover, with larger number of the particles, the adaptive RB construction leads to similar number of reduced basis functions for the same initial tolerance, and achieves
higher speedup since the RB construction time does not change much. Furthermore, compared to the fixed RB construction with relatively small tolerance at the initial step, adaptive RB leads to higher speedup, while achieving higher accuracy than the former as shown in Fig.\ \ref{fig:accuracy-uniform}. We note that the reduced basis (averaged) approximation (of $\eta_y^h$) errors are smaller than the HiFi (averaged) approximation errors of about $10^{-4}$ at SVGD step $l = 99$, even with the initial tolerance $\varepsilon_r^0 = 0.1$ as can be observed from Fig.\ \ref{fig:rom-update} and in particular for $\varepsilon_r^0 = 0.01$ as seen from Fig.\ \ref{fig:accuracy-uniform}. Therefore, the adaptive RB may achieve much higher speedup if the high-fidelity approximation is refined to achieve the same errors as those of the reduced basis approximations.

\section{Conclusion}
\label{sec:conclusion}

In this work, we developed and analyzed a computational approach for efficient sampling from the posterior distribution in Bayesian inversion governed by PDEs. This approach builds on an adaptive integration of optimization based SVGD for sampling and goal-oriented RB approximations for computational efficiency, leading to the combined advantages of (i) optimization for sampling using local geometric information (gradient) of the posterior, (ii) goal-oriented, certified, and improved RB approximations of the potential and its gradient at all SVGD samples due to the dual-weighted residual and its gradient used in an adaptive greedy construction, and (iii) a natural balance between computational accuracy and efficiency by adaptive tuning of the tolerance for the RB construction according to the convergence of the SVGD sampling. 
We carried out detailed analysis of the RB approximation errors of the potential and its gradient, which have been demonstrated in the numerical examples, the induced errors of the posterior distribution measured by Kullback--Leibler divergence, as well as the errors induced in the samples. In particular, we proved and demonstrated that the improved RB approximations for both the potential and its gradient by the dual-weighted residual and its gradient achieve higher accuracy than the RB approximations. Moreover, we showed that the proposed SVRB method achieved over 100X speedup compared to SVGD for PDE-constrained Bayesian inversion with both uniform and Gaussian random parameters, with the speedup becoming more pronounced as the number of samples increases.

Several avenues for the further development of SVRB include: (i) The development and analysis of SVRB is built on linear PDEs with affine dependence of the parameter. Extension of this method and its analysis for nonlinear and nonaffine PDE models is under  investigation. (ii) Beyond using gradient information of the posterior in SVGD, we can exploit its Hessian and use the Stein variational Newton (SVN) method \cite{DetommasoCuiMarzoukEtAl18} to improve the convergence of sampling and SVRB. (iii) In addition to the RB dimension reduction in the physical (high-fidelity approximation) space, we can integrate a simultaneous dimension reduction in both the physical space and the parameter space to tackle the curse-of-dimensionality faced by Stein variational methods for Bayesian inference of high-dimensional parameters, by leveraging the projected SVGD/SVN \cite{ChenWuChenEtAl19, ChenGhattas20}. (iv) A more comprehensive analysis including the convergence w.r.t.\ the number of samples and variational iterations, combined with the RB errors, is theoretically interesting. (v) It is important to further develop a hybrid parallel implementation in both sampling and PDE solving for large-scale Bayesian inversion. 
 
%

\bibliographystyle{siamplain}
\bibliography{references}

\appendix

\section{Error estimates}

\subsection{Proof of Lemma \ref{lemma:nabla-u-phi}}
\label{sec:proof-lemma-1}
\begin{proof}
	For any $\theta \in \Theta$, and any $j = 1, \dots, d$, let $\varepsilon > 0$ be such that $\theta + \varepsilon e_j \in \Theta$ for the vector $e_j \in \bR^d$ with $i$-th element $(e_j)_i = \delta_{ij}$, being $\delta_{ij}$ the Kronecker delta. By subtracting \eqref{eq:state} at $\theta$ from it at $\theta + \varepsilon e_j$, we have
	\begin{equation}
	A(u(\theta + \varepsilon e_j), v; \theta + \varepsilon e_j) - A(u(\theta), v; \theta) = F(v; \theta + \varepsilon e_j) - F(v; \theta), \quad \forall v \in V,
	\end{equation}
	which, by dividing by $\varepsilon$ and letting $\varepsilon \to 0$, together with the continuous differentiability of $A(\cdot, \cdot; \theta)$ and $F(\cdot; \theta)$ w.r.t.\ $\theta$ under A2 of Assumption \ref{ass:AF}, leads to 
	\begin{equation}\label{eq:state-partial-theta}
	A(\partial_{\theta_j} u(\theta), v; \theta ) + \partial_{\theta_j} A(u(\theta), v; \theta) = \partial_{\theta_j} F(v;\theta), \quad \forall v \in V.
	\end{equation}
	where we denote 
	\begin{equation}
	\partial_{\theta_j}u(\theta) := \lim_{\varepsilon \to 0} \frac{u(\theta + \varepsilon e_j) - u(\theta)}{\varepsilon},
	\end{equation}
	which exists and is the unique solution of \eqref{eq:state-partial-theta} by Lax--Milgram theorem under of Assumption \ref{ass:AF}. Moreover, we have the stability estimate 
	\begin{equation}
	||\partial_{\theta_j} u(\theta) ||_V \leq \frac{1}{\alpha(\theta)} (\rho_j(\theta) ||u(\theta)||_V + ||\partial_{\theta_j} F(\cdot;\theta)||_{V'}),
	\end{equation}
	which leads to \eqref{eq:partial-theta-u} by the stability estimate \eqref{eq:stability-u}. 
	
	Similarly, by subtracting \eqref{eq:adjoint} at $\theta$ from it at $\theta + \varepsilon e_j$, we have 
	\begin{equation}
	A(w, \psi(\theta + \varepsilon e_j);\theta+\varepsilon e_j) - A(w, \psi(\theta);\theta) = - \nabla_u \eta_y|_{u(\theta + \varepsilon e_j)}(w) + \nabla_u \eta_y|_{u(\theta)}(w) 
	\end{equation}
	which, by dividing by $\varepsilon$ and letting $\varepsilon \to 0$, together with the continuous differentiability of $A(\cdot, \cdot; \theta)$ under A2 of Assumption \ref{ass:AF}, leads to 
	\begin{equation}\label{eq:adjoint-partial-theta}
	A(w, \partial_{\theta_j} \psi(\theta); \theta) + \partial_{\theta_j} A(w, \psi(\theta);\theta) = (y - \cO(\partial_{\theta_j} u(\theta)))^T \Gamma^{-1} \cO(w),
	\end{equation}
	where we denote 
	\begin{equation}
	\partial_{\theta_j} \psi(\theta): = \lim_{\varepsilon \to 0} \frac{\psi(\theta + \varepsilon e_j) - \psi(\theta)}{\varepsilon},
	\end{equation}
	which is the unique solution of \eqref{eq:adjoint-partial-theta} by Lax--Milgram theorem under Assumption \ref{ass:AF}, and satisfies the stability estimate
	\begin{equation}
	||\partial_{\theta_j} \psi(\theta)||_V \leq \frac{1}{\alpha(\theta)} (\rho_j(\theta) ||\psi(\theta)||_V + C_y + C_\cO ||\partial_{\theta_j} u(\theta)||_V),
	\end{equation}
	where $C_y$ and $C_\cO$ are defined in \eqref{eq:Gamma-cO-y}. This leads to \eqref{eq:partial-theta-psi} by the stability estimate \eqref{eq:stability-psi}. By following the same argument, we can show that the estimate \eqref{eq:partial-theta-u} holds for $\nabla_\theta u_h(\theta)$ and $\nabla_\theta u_r(\theta)$, while \eqref{eq:partial-theta-psi} holds for $\nabla_\theta \psi_h(\theta) $ and $\nabla_\theta \psi_r(\theta) $.
\end{proof}

\subsection{Error estimates for the state $u$ and adjoint $\psi$ and their gradients $\nabla_\theta u(\theta)$ and $\nabla_\theta \psi(\theta)$}
\label{sec:upsigrad}
Let $R_u(u_r, \cdot; \theta)$ and $R_\psi(\psi_r, \cdot; \theta)$ denote the residuals of the state and adjoint equations 
\begin{equation}\label{eq:state-residual}
R_u(u_r, v_h; \theta) = A(u_r, v_h; \theta) - F(v_h;\theta) \quad \forall v_h \in V_h,
\end{equation}
and 
\begin{equation}\label{eq:adjoint-residual}
R_\psi(w_h, \psi_r; \theta) = A(w_h, \psi_r; \theta) + \nabla_u \eta_y|_{u_r}(w_h) \quad \forall w_h \in V_h.
\end{equation}
The following error estimates can be obtained for the reducal basis errors $e_r^u$ and $e_r^\psi$. 

\begin{lemma}\label{lemma:bound-u-psi-h-r}
	Under Assumption \ref{ass:AF}, for any $\theta \in \Theta$, there holds 
	\begin{equation}
	||e_r^u(\theta)||_V \leq \frac{1}{\alpha(\theta)} ||R_u(u_r,\cdot;\theta)||_{V'},
	\end{equation}
	and 
	\begin{equation}
	||e_r^\psi(\theta)||_V \leq \frac{1}{\alpha(\theta)} ||R_\psi(\cdot, \psi_r;\theta)||_{V'} + \frac{C_\cO}{\alpha(\theta)} ||e_r^u(\theta)||_V.
	\end{equation}
\end{lemma}

\begin{proof}
	By the high-fidelity state equation \eqref{eq:state-HiFi}, the residual \eqref{eq:state-residual} becomes 
	\begin{equation}
	R_u(u_r, v_h; \theta) = A(u_r, v_h; \theta) - A(u_h, v_h; \theta) = - A(e_r^u, v_h; \theta).
	\end{equation}
	By replacing $v_h = e_r^u$, and using Assumption \ref{ass:AF} we have 
	\begin{equation}
	\alpha(\theta)||e_r^u||_V^2 \leq A(e_r^u, e_r^u; \theta) = -R_u(u_r, e_r^u; \theta) \leq ||R_u(u_r, \cdot;\theta)||_{V'} ||e_r^u||_V,
	\end{equation}
	which concludes for the estimate of $e_r^u$. To estimate $e_r^\psi$, we have 
	\begin{equation}
	R_\psi(w_h, \psi_r ;\theta) = A(w_h, \psi_r;\theta) + \nabla_u\eta_y|_{u_h}(w_h) + (\nabla_u\eta_y|_{u_r}(w_h) - \nabla_u\eta_y|_{u_h}(w_h)),
	\end{equation}
	which, by the adjoint high-fidelity equation \eqref{eq:adjoint-HiFi} and $\nabla_u \eta_y$ in \eqref{eq:nabla_u-eta}, gives 
	\begin{equation}
	-A(w_h, e_r^\psi; \theta) = R_\psi(w_h, \psi_r;\theta) + \cO(e_r^u) \Gamma^{-1} \cO(w_h).
	\end{equation}
	By replacing $w_h = e_r^\psi$, under Assumption \ref{ass:AF}, we have 
	\begin{equation}
	\begin{split}
	\alpha(\theta)||e_r^\psi||_V^2 & \leq A(e_r^\psi, e_r^\psi; \theta)  \\
	& = -R_\psi(w_h, \psi_r;\theta) - \cO(e_r^u) \Gamma^{-1} \cO(e_r^\psi)\\
	& \leq ||R_\psi(\cdot, \psi_r; \theta)||_{V'}||e_r^\psi||_V + C_\cO ||e_r^u||_V||e_r^\psi||_V,
	\end{split}
	\end{equation}
	which concludes.
\end{proof}

For $j = 1, \dots, d$, we denote the solutions of the high-fidelity and reduced basis approximations of \eqref{eq:partial-theta-u} in spaces $V_h$ and $V_r$ as $\partial_{\theta_j} u_h(\theta)$ and $\partial_{\theta_j} u_r(\theta)$, the solutions of the high-fidelity and reduced basis approximations of \eqref{eq:partial-theta-psi} in spaces $V_h$ and $W_r$ as $\partial_{\theta_j} \psi_h(\theta)$ and $\partial_{\theta_j} \psi_r(\theta)$, respectively.
Let $R_u^j(u_r,v_h;\theta)$ and $R_\psi^j(w_h, \psi_r;\theta)$ denote the residuals of \eqref{eq:partial-theta-u} and \eqref{eq:partial-theta-psi} as
\begin{equation}
R_u^j(u_r,v_h;\theta) = A(\partial_{\theta_j} u_r(\theta), v_h; \theta ) + \partial_{\theta_j} A(u_r(\theta), v_h; \theta) - \partial_{\theta_j} F(v_h;\theta) \quad \forall v_h \in V_h,
\end{equation}
and 
\begin{equation}
\begin{split}
R_\psi^j(w_h, \psi_r;\theta) & = A(w_h, \partial_{\theta_j} \psi_r(\theta); \theta) + \partial_{\theta_j} A(w_h, \psi_r(\theta);\theta) \\
& - (y - \cO(\partial_{\theta_j} u_r(\theta)))^T \Gamma^{-1} \cO(w_h)\quad \forall w_h \in V_h.
\end{split}
\end{equation}
\begin{lemma}
	Under Assumption \ref{ass:AF}, for any $\theta \in \Theta$, we have 
	\begin{equation}
	||\nabla_\theta e_r^u(\theta)||_{V^d} \leq \frac{1}{\alpha(\theta)} \sum_{j=1}^d ||R_u^j( u_r, \cdot; \theta)||_{V'} +\frac{1}{\alpha(\theta)} \sum_{j=1}^d \rho_j(\theta) ||e_r^u(\theta)||_V，
	\end{equation}
	and 
	\begin{equation}
	\begin{split}
	||\nabla_\theta e_r^\psi(\theta)||_{V^d} & \leq \frac{1}{\alpha(\theta)} \sum_{j=1}^d ||R_\psi^j(\cdot, \psi_r ; \theta)||_{V'} \\
	& + \frac{1}{\alpha(\theta)} \sum_{j=1}^d \rho_j(\theta) ||e_r^\psi(\theta)||_V  + \frac{C_\cO}{\alpha(\theta)} ||\nabla_\theta e_r^u(\theta)||_{V^d}.
	\end{split}
	\end{equation}
\end{lemma}

\begin{proof}
	By definition of the residual $R_\psi^j(w_h, \psi_r;\theta)$, we have 
	\begin{equation}
	R_\psi^j(w_h, \psi_r;\theta) = -A(\partial_{\theta_j} e_r^u(\theta), v_h;\theta) - \partial_{\theta_j} A(e_r^u(\theta), v_h;\theta),
	\end{equation}
	which, by replacing $v_h = \partial_{\theta_j} e_r^u$, and using Assumption \ref{ass:AF}, leads to 
	\begin{equation}
	\alpha(\theta) ||\partial_{\theta_j} e_r^u(\theta)||_V \leq ||R^j_u(u_r, \cdot;\theta)||_{V'} + \rho_j(\theta) ||e_r^u(\theta)||_V,
	\end{equation}
	which concludes the first estimate by summing over $j = 1, \dots, d$. Similarly, by definition of $R_\psi^j(w_h, \psi_r;\theta)$, we have 
	\begin{equation}
	R_\psi^j(w_h, \psi_r;\theta) = -A(w_h, \partial_{\theta_j} e_r^\psi(\theta);\theta) - \partial_{\theta_j} A(w_h, e_r^\psi(\theta);\theta) - \cO(\partial_{\theta_j} e_r^u(\theta)) \Gamma^{-1} \cO(w_h),
	\end{equation}
	which, by replacing $w_h = \partial_{\theta_j} e_r^\psi(\theta)$ and using Assumption \ref{ass:AF}, leads to 
	\begin{equation}
	\alpha(\theta) ||\partial_{\theta_j} e_r^\psi(\theta)||_V \leq ||R_\psi^j(\cdot, \psi_r ; \theta)||_{V'} + \rho_j(\theta) ||e_r^\psi(\theta)||_V + C_\cO ||\nabla_{\theta_j} e_r^u(\theta)||_V,
	\end{equation}
	which concludes the second estimate by summing over $j = 1, \dots, d$.
\end{proof}

\subsection{Error estimates for SVGD samples}
\label{sec:estimates4samples}
Let $\theta_n^{0,h} = \theta_n^{0,r} = \theta_n^{0,\Delta} = \theta_n$, $n = 1, \dots, M$, denote the samples drawn from the prior distribution. Let $\theta_n^{l,h}$, $\theta_n^{l,r}$, and $\theta_n^{l,\Delta}$, $l = 1, \dots, L$, $n = 1, \dots, M$, denote the samples obtained by the SVGD Algorithm \ref{alg:SVGD}, and $\alpha_l^h$, $\alpha_l^r$, $\alpha_l^\Delta$, $l = 0, \dots, L-1$, denote the step sizes with $\eta_y$ approximated by $\eta_y^h$, $\eta_y^r$, and $\eta_y^\Delta$, respectively. 

\begin{assumption}
We assume that there exists a constant $C_p$ such that 
\begin{equation}\label{eq:Lips-p-0}
\left|\left|\frac{\nabla_\theta p_0(\theta)}{p_0(\theta)} - \frac{\nabla_\theta p_0(\theta')}{p_0(\theta')} \right|\right|_1 
\leq 
C_p ||\theta - \theta'||_1.
\end{equation}
Moreover, there exists a constant $C_\eta$ such that 
\begin{equation}\label{eq:Lips-eta}
||\nabla_\theta \eta_y^h(\theta) - \nabla_\theta \eta_y^h(\theta')||_1 \leq C_\eta ||\theta - \theta'||_1.
\end{equation}
\end{assumption}

\begin{theorem}
	Under Assumption \ref{ass:AF}, and assume that the step sizes $\alpha_i^h = \alpha_i^r$ for $i = 0, \dots, l$, then we have 
	\begin{equation}
	||\theta_n^{l+1, h} - \theta_n^{l+1, r} ||_1 \leq \sum_{i=0}^{l} C_i \frac{1}{M}\sum_{m=1}^M ||\nabla_\theta e_r^\eta(\theta_m^{i,r})||_1,
	\end{equation}
	for some constants $C_i$, $i = 0, \dots, l$.
\end{theorem}

\begin{proof}
	At $l = 0 $, by the definition of the samples \eqref{eq:samples} and \eqref{eq:pertubation}, under the assumptions $\theta_n^{0,h} = \theta_n^{0,r} =\theta_n^0$, $n = 1, \dots, M$, and $\alpha_0^h = \alpha_0^r = \alpha_0$,
	we have 
	\begin{equation}
	\theta_n^{1, h} - \theta_n^{1, r} = \alpha_0 \frac{1}{M} \sum_{m = 1}^M (\nabla_\theta \eta_y^r(\theta_m^{0})-\nabla_\theta \eta_y^h(\theta_m^{0})) k(\theta_m^{0}, \theta_n^{0}),
	\end{equation}
	which can be bounded by 
	\begin{equation}
	||\theta_n^{1, h} - \theta_n^{1, r}||_1 \leq \alpha_0 \frac{1}{M} \sum_{m=1}^M ||\nabla_\theta e_r^\eta(\theta_m^{0})||_1, \quad n = 1, \dots, M.
	\end{equation}
	For $l > 0$, we have 
	\begin{equation}
	\begin{split}
	\theta_n^{l+1, h} - \theta_n^{l+1, r} & = \theta_n^{l, h} - \theta_n^{l, r} \\
	& +  \alpha_l \frac{1}{M} \sum_{m = 1}^M \nabla_\theta \log(p_y^h(\theta_m^{l,h})) k(\theta_m^{l,h}, \theta_n^{l, h}) + \nabla_\theta k(\theta_m^{l,h}, \theta_n^{l, h})\\
	& -  \alpha_l \frac{1}{M} \sum_{m = 1}^M \nabla_\theta \log(p_y^r(\theta_m^{l,r})) k(\theta_m^{l,r}, \theta_n^{l, r}) + \nabla_\theta k(\theta_m^{l,r}, \theta_n^{l, r}).
	\end{split}
	\end{equation}
	To bound the error, we first consider 
	\begin{equation}
	\begin{split}
	& \nabla_\theta \log(p_y^h(\theta_m^{l,h})) k(\theta_m^{l,h}, \theta_n^{l, h}) - \nabla_\theta \log(p_y^r(\theta_m^{l,r})) k(\theta_m^{l,r}, \theta_n^{l, r})\\
	& \leq || \nabla_\theta \log(p_y^h(\theta_m^{l,h})) - \nabla_\theta \log(p_y^r(\theta_m^{l,r}))||_1 k(\theta_m^{l,h}, \theta_n^{l, h}) \\
	& + ||\nabla_\theta \log(p_y^r(\theta_m^{l,r}))||_1 |k(\theta_m^{l,h}, \theta_n^{l, h}) - k(\theta_m^{l,r}, \theta_n^{l, r})|.
	\end{split} 
	\end{equation}
	By definition, we have
	\begin{equation}
	\begin{split}
	& \nabla_\theta \log(p_y^h(\theta_m^{l,h})) - \nabla_\theta \log(p_y^r(\theta_m^{l,r})) \\
	& = -\nabla_\theta \eta_y^h(\theta_m^{l,h}) + \nabla_\theta \eta_y^r(\theta_m^{l,r}) + \left(
	\frac{\nabla_\theta p_0(\theta_m^{l,h})}{p_0(\theta_m^{l,h})} - \frac{\nabla_\theta p_0(\theta_m^{l,r})}{p_0(\theta_m^{l,r})},
	\right)
	\end{split}
	\end{equation}
	where by assumption \eqref{eq:Lips-p-0}, the second term can be bounded by 
	\begin{equation}\label{eq:p0-theta-h-r}
	\left|\left|\frac{\nabla_\theta p_0(\theta_m^{l,h})}{p_0(\theta_m^{l,h})} - \frac{\nabla_\theta p_0(\theta_m^{l,r})}{p_0(\theta_m^{l,r})} \right|\right|_1 
	\leq 
	C_0 ||\theta_m^{l,h} - \theta_m^{l,r}||_1.
	\end{equation}
	The first term can be written as 
	\begin{equation}
	-\nabla_\theta \eta_y^h(\theta_m^{l,h}) + \nabla_\theta \eta_y^r(\theta_m^{l,r}) = -\nabla_\theta \eta_y^h(\theta_m^{l,h}) + \nabla_\theta \eta_y^h(\theta_m^{l,r}) -  \nabla_\theta \eta_y^h(\theta_m^{l,r}) +  \nabla_\theta \eta_y^r(\theta_m^{l,r}),
	\end{equation}
	which can be bounded by 
	\begin{equation}\label{eq:eta-theta-h-r}
	||-\nabla_\theta \eta_y^h(\theta_m^{l,h}) + \nabla_\theta \eta_y^r(\theta_m^{l,r})||_1 
	\leq 
	C_\eta ||\theta_m^{l,h} - \theta_m^{l,r}||_1 +  ||\nabla e_r^\eta(\theta_m^{l,r})||_1,
	\end{equation}
	where in the first term we used the assumption \eqref{eq:Lips-eta}. 
	Therefore, combining \eqref{eq:p0-theta-h-r} and \eqref{eq:eta-theta-h-r} we obtain 
	\begin{equation}
	||\nabla_\theta \log(p_y^h(\theta_m^{l,h})) - \nabla_\theta \log(p_y^r(\theta_m^{l,r}))||_1 \leq (C_p+C_\eta) ||\theta_m^{l,h} - \theta_m^{l,r}||_1 +  ||\nabla e_r^\eta(\theta_m^{l,r})||_1.
	\end{equation}
	Moreover, by definition of the kernel $k(\theta, \theta')$ in \eqref{eq:kernel}, it is easy to show that  
	\begin{equation}
	|k(\theta_m^{l,h}, \theta_n^{l,h}) - k(\theta_m^{l,r}, \theta_n^{l,r})| \leq C_{k,1} (||\theta_m^{l,h} - \theta_m^{l,r}||_1 + ||\theta_n^{l,h} - \theta_n^{l,r}||_1),
	\end{equation}
	and
	\begin{equation}
	||\nabla_\theta k(\theta_m^{l,h}, \theta_n^{l,h}) - \nabla_\theta k(\theta_m^{l,r}, \theta_n^{l,r})||_1 \leq C_{k,2} (||\theta_m^{l,h} - \theta_m^{l,r}||_1 + ||\theta_n^{l,h} - \theta_n^{l,r}||_1),
	\end{equation}
	for some constants $C_{k,1}$ and $C_{k,2}$. A combination of the above bounds concludes.
	
	
\end{proof}

\section{Numerical example with Gaussian prior distribution}
\label{sec:Gaussain}


In this example, we consider the parameter $\theta = (\theta_1, \dots, \theta_9)$ with the $d = 9$ random components obeying i.i.d.\ Gaussian distribution $\cN(0,1)$. The coefficients are set as $c_j(\theta) = e^{\theta_j/2}$ with $j = 1, \dots, 9$. The basis functions are given by 
$a_0 = 0$ and $a_j = \chi_{D_j}$, where $\chi_{D_j}$ is a characteristic function with $\chi_{D_j}(x) = 1$ for $x \in D_j$ and $\chi_{D_j}(x) = 0$ otherwise. $D_j$, $j = 1, \dots, 9$, are a uniform square partition of the domain $D$ each with area $1/9$, ordered from left to right, bottom to top.

\begin{figure}[!htb]
	\centering
	\includegraphics[trim=0 0 40 40,clip=True,width=0.48\textwidth]{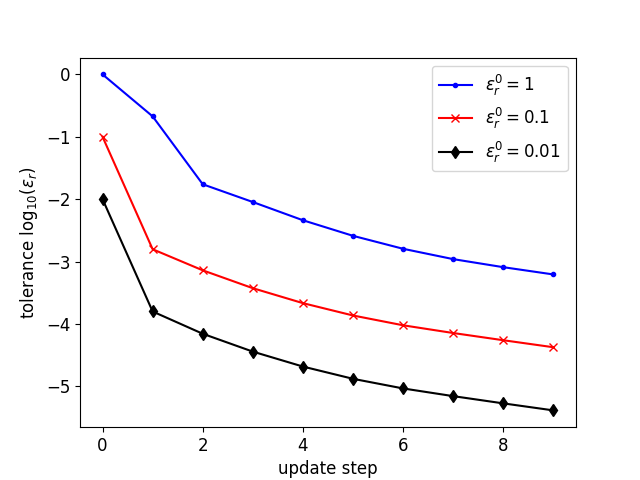}
	\includegraphics[trim=0 0 40 40,clip=True,width=0.48\textwidth]{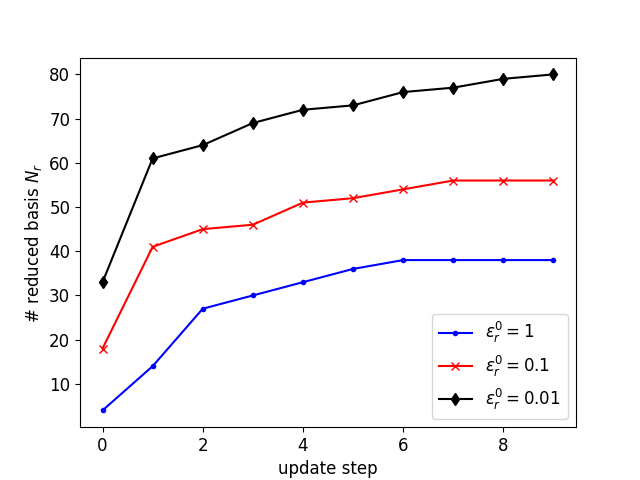}
	\caption{Change of tolerances (left) and the number of reduced basis functions (right) w.r.t.\ the RB update step $i$ with SVGD step $l = i K$ and $K = 20$ in Algorithm \ref{alg:adaptive-greedy}.}
	\label{fig:rom-update-lognormal}
\end{figure}

We run Algorithm \ref{alg:adaptive-greedy} for the construction of reduced basis approximations and their applications in the SVGD process. We set the initial tolerance as $\varepsilon_r^0 = 1, 0.1, 0.01$, respectively, and update the reduced basis approximations every $K = 20$ SVGD steps with new tolerance given by $\varepsilon_r^l = \varepsilon_r^0 t_l$ with $t_l$ defined in \eqref{eq:t-l}. The changes of the tolerances and the number of reduced basis functions for different initial tolerances are shown in Fig.\ \ref{fig:rom-update-lognormal}, from which we can see that as the gradient norm of SVGD update decreases, i.e., the particles become closer to following the posterior distribution, the reduced basis approximations become more accurate with larger number of reduced basis functions.

\begin{figure}[!htb]
	\centering
	\begin{subfigure}{0.59\textwidth}
		\centering
		\centering
		\includegraphics[trim=80 10 80 30,clip=True,width=0.32\textwidth]{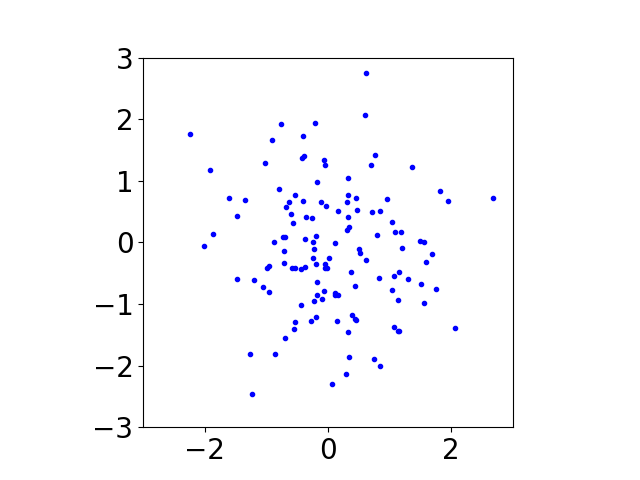}
		\includegraphics[trim=80 10 80 30,clip=True,width=0.32\textwidth]{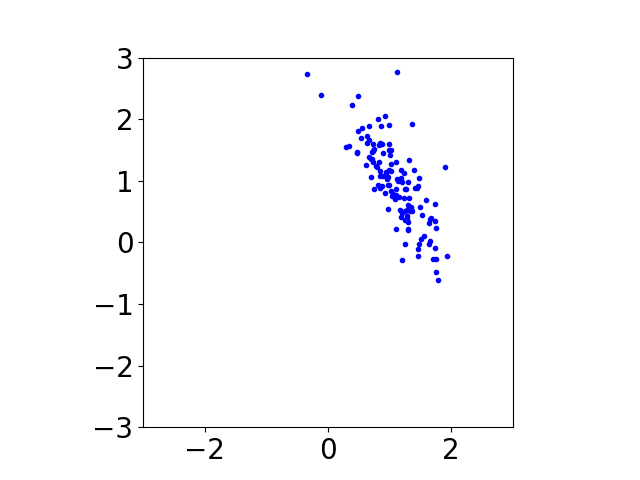}
		\includegraphics[trim=80 10 80 30,clip=True,width=0.32\textwidth]{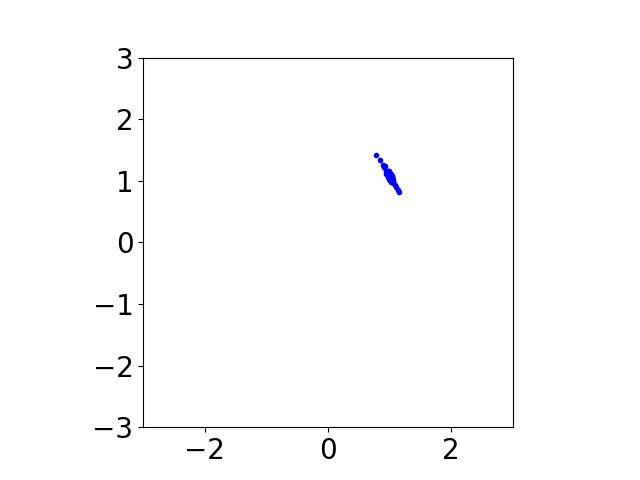}
		
		\includegraphics[trim=80 10 80 30,clip=True,width=0.32\textwidth]{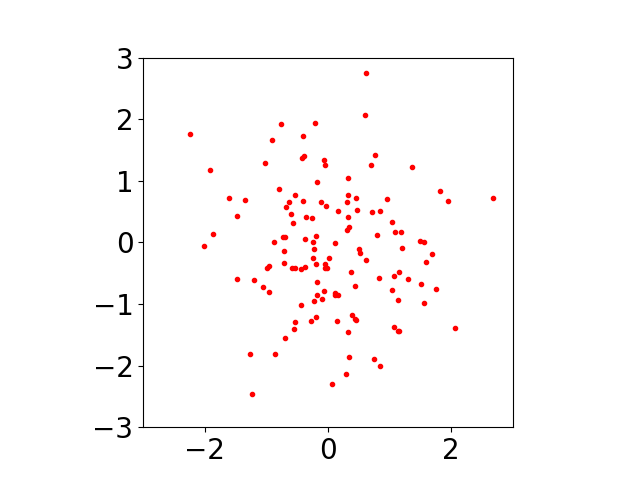}
		\includegraphics[trim=80 10 80 30,clip=True,width=0.32\textwidth]{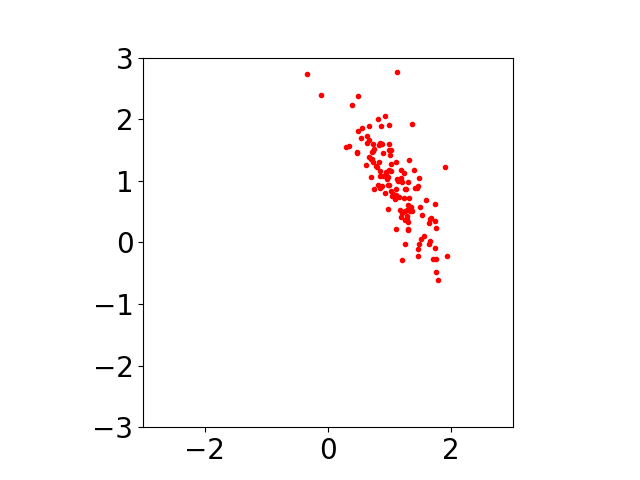}
		\includegraphics[trim=80 10 80 30,clip=True,width=0.32\textwidth]{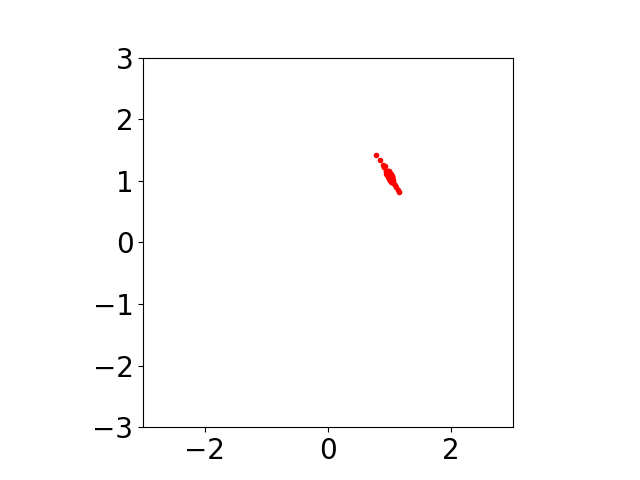}
		\caption{Locations of 128 particles ($\theta_1, \theta_2$) at SVGD step $l = 0$ (left), $19$ (middle), $199$ (right) by high-fidelity (top) and reduced basis (bottom) approximations.}
	\end{subfigure}%
	~ 
	\begin{subfigure}{0.39\textwidth}
		\centering
		\includegraphics[trim=70 10 80 30,clip=True,width=0.96\textwidth]{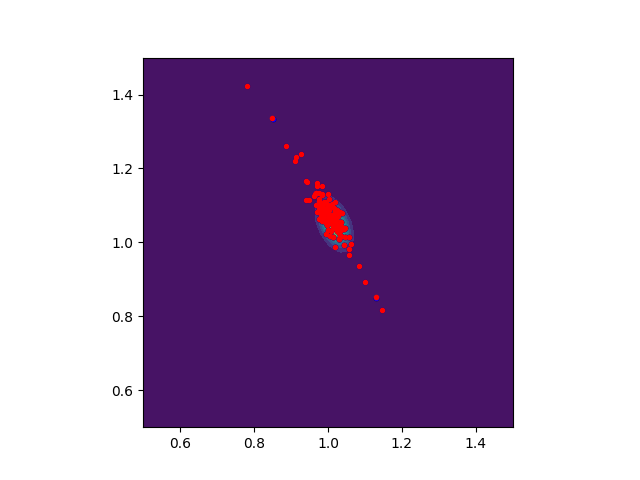}
		\caption{Contour of the marginal posterior density for $(\theta_1, \theta_2)$ and locations of particles at SVGD step $l = 199$.}
	\end{subfigure}
	\caption{Comparision of particles by high-fidelity and reduced basis approximations.}
	\label{fig:particles-lognormal}
\end{figure}

Fig.\ \ref{fig:particles-lognormal} depicts the update of 128 particles (projected in dimension ($\theta_1, \theta_2$)) by SVGD with high-fidelity and reduced basis approximations (with initial tolerance $\varepsilon_r^0 = 0.01$) of the PDE models, respectively. At the initial step $l = 0$, we randomly draw 128 samples from the Gaussian prior distribution, as shown in the left two figures of part (a), and use them for both the high-fidelity and reduced basis approximations. At SVGD step $l = 19$ and $l = 199$, the updated particles with different approximations are displayed in the middle and right two figures of part (a), which appear very close to each other. More details are shown in part (b) in the enlarged region where the marginal posterior density of high-fidelity approximation in dimension ($\theta_1, \theta_2$) is evidently different from zero, from which we can see that the particles obtained by the reduced basis approximations are very close to those by the high-fidelity approximations, which both have effectively good empirical representation of the posterior distribution. 

\begin{figure}[!htb]
	\centering
	\includegraphics[trim=0 0 40 40,clip=True,width=0.48\textwidth]{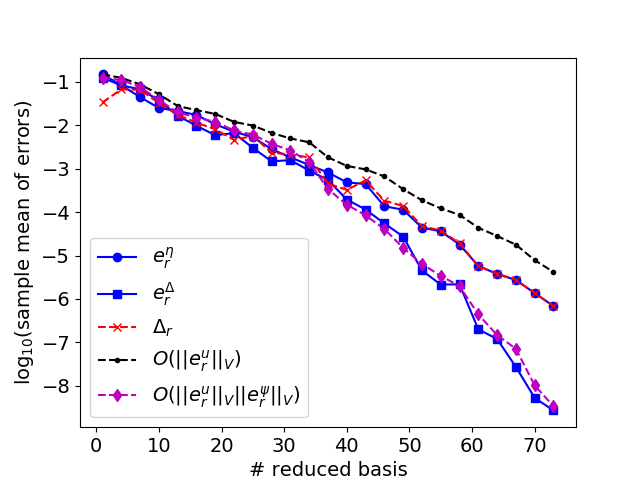}
	\includegraphics[trim=0 0 40 40,clip=True,width=0.48\textwidth]{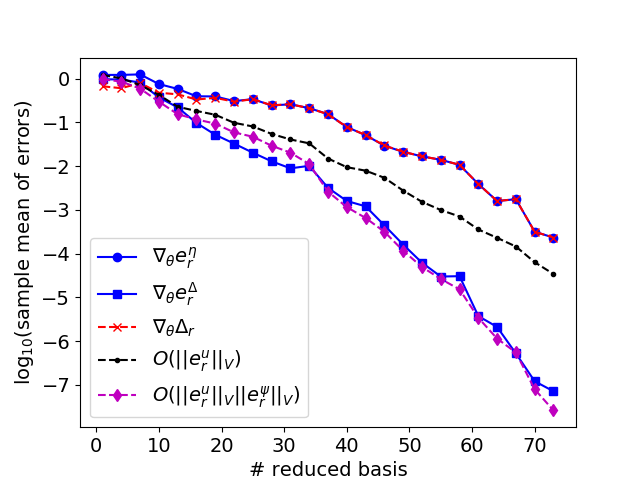}
	
	\includegraphics[trim=0 0 40 40,clip=True,width=0.48\textwidth]{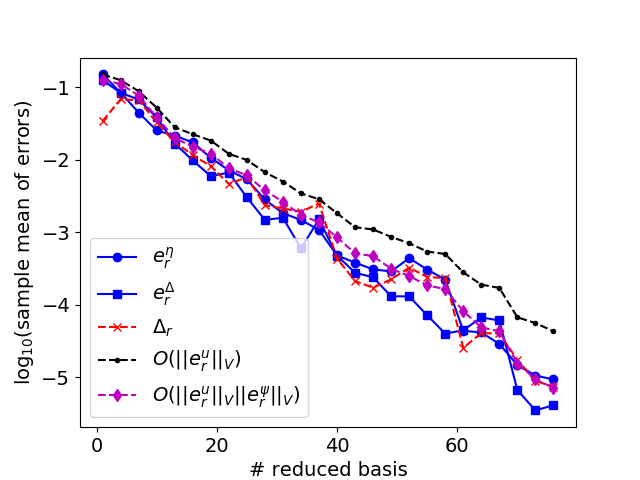}
	\includegraphics[trim=0 0 40 40,clip=True,width=0.48\textwidth]{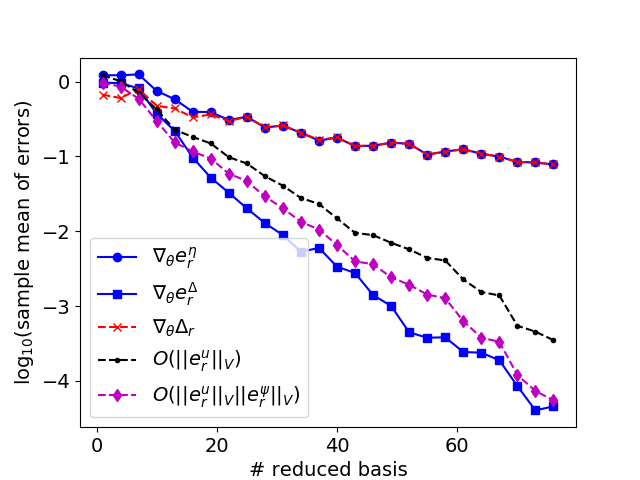}
	\caption{Sample mean of the approximation errors and estimates for adaptive RB (top) and fixed RB (bottom) approximations of $\eta_y^h$ (left) and $\nabla_\theta \eta_y^h$ (right) at step $l = 199$.}
	\label{fig:accuracy-lognormal}
\end{figure}

Fig.\ \ref{fig:accuracy-lognormal} demonstrates the accuracy of the reduced basis approximations of the potential $\eta_y$ and its gradient $\nabla_\theta \eta_y$, the efficacy of the error estimate $\Delta_r$ used in the greedy algorithm and various error bounds, as well as the advantage of the adaptive greedy construction. More specifically, from the top two figures on the decay of the sample averaged reduced basis approximation errors $e_r^\eta$ and $e_r^\Delta$ at SVGD step $l = 99$, obtained by the adaptive greedy Algorithm \ref{alg:adaptive-greedy}, we can see that the averaged error $e_r^\eta$ decay asymptotically as the averaged error bound $||e_r^u||_V$ (which is rescaled by a constant such that the error and the bound are equal at $N_r = 1$), as predicted by Lemma \ref{lemma:bound-eta-h-r}. Moreover, the averaged error $e_r^\Delta$ decay asymptotically as the averaged error bound $||e_r^u||_V ||e_r^\psi||_V$, as predicted by Lemma \ref{lemma:bound-eta-h-r} where we note that $||e_r^u||_V ||e_r^\psi||_V$ dominates $||e_r^u||_V^2$. By comparison of $e_r^\eta$ and $e_r^\Delta$, we can also see that, $\eta_y^\Delta$, the reduced basis approximation of the potential, $\eta_y^r$, corrected by the dual weighted residual $\Delta_r$, is much more accurate than $\eta_y^r$ itself, especially when the number of reduced basis functions becomes large. This observation can be confirmed by the closeness of the residual $\Delta_r$ and the error $e_r^\eta$ as shown in the top-left part of Fig. \ref{fig:accuracy-lognormal}. Similar conclusion in alignment with Lemma \ref{lemma:grad-e-r-eta} can be drawn for the reduced basis approximation of the gradient of the potential $\nabla_\theta \eta_y^h$, as depicted in the top-right part of the figure. Note that we did not compute the norm of the gradients of the state and adjoint, i.e., $||\nabla_\theta e_r^u||_{V^d}$ and $||\nabla_\theta e_r^\psi||_{V^d}$, as they involve solving additional $2d$ PDE problems presented in Section \ref{sec:upsigrad}, which are not needed in the adaptive greedy Algorithm \ref{alg:adaptive-greedy}. The bottom two figures of Fig.\ \ref{fig:accuracy-lognormal} show the decay of the errors and bounds for reduced basis approximation with the fixed reduced basis functions constructed at the initial step of SVGD with tolerance $\varepsilon_r = 10^{-5}$, in contrast with the adaptive construction. We can see that the reduced basis approximations constructed once and used for all later SVGD evaluations become less accurate than the reduced basis approximations by the adaptive construction, both for the approximation of $\eta_y$ and $\nabla_\theta \eta_y$, even when the number of reduced basis functions of the former is much larger than the latter. This demonstrates the advantage of the adaptive greedy construction in terms of accuracy of the reduced basis approximations.

\begin{table}[h!]
	\centering
	\begin{tabular}{| c | c | c | c c c | c|} 
		\hline
		\multicolumn{2}{|c|}{} & HiFi & \multicolumn{3}{|c|}{adaptive RB} & fixed RB\\  
		\hline
		\multicolumn{2}{|c|}{initial tolerance $\varepsilon_r^0$} & n/a & $1$ &  $0.1$ &  $0.01$ & $0.00001$ \\
		\hline
		\multirow{4}{*}{$M=128$}&DOF ($N_h, N_r$) & 16641 
		& 39 & 52 & 71 & 81\\
		& time to build RB & n/a & $11.8$ & $16.9$ & $25.8$ & $29.3$\\
		& time for evaluation & $8.4\times 10^3$ & $29.2$ & $32.8$ & $40.0$ & $49.2$ \\
		& speedup factor & 1 & 205 & 169 & 128 & 107\\
		\hline
		\multirow{4}{*}{M=256}&DOF ($N_h, N_r$) & 16641 
		& 38 & 56 & 80 & 86 \\
		& time to build RB & n/a & $13.4$ & $21.9$ & $36.2$ & $36.7$ \\
		& time for evaluation & $1.7\times 10^4$ & $58.0$ & $67.9$ & $83.4$ & $103.1$\\
		& speedup factor & 1 & 235 & 187 & 140 & 120\\
		\hline
	\end{tabular}
	\caption{Comparison of computational cost of high-fidelity (HiFi) and reduced basis (RB) approximations for SVGD up to $l = 199$, with different number of particles, different RB construction schemes and tolerances, in terms of degrees of freedom (DOF), CPU time for evaluation and RB construction, and speedup factor, which is the ratio of HiFi evaluation time/(RB construction + evaluation time).}
	\label{table:cost-lognormal}
\end{table}

We report the computational cost of high-fidelity and reduced basis approximations in the SVGD process up to step $l = 199$ in Table \ref{table:cost-lognormal}. For the high-fidelity approximation, the degrees of freedom is 16,641 by the piecewise linear elements in a triangle mesh of size $129\times 129$, which leads to an averaged (128 samples at SVGD step $l = 199$) approximation error for the potential $\eta_y$ at about $10^{-4}$ (using mesh size $257\times 257$ as reference). From the results we see that with increasing initial tolerance $\varepsilon_r^0 = 1, 0.1, 0.01$, the adaptive RB  becomes more expensive for construction and evaluation, which achieves mostly over 100X speedup compared to the HiFi in terms of CPU time. Moreover, with larger number of the particles, the adaptive RB construction leads to similar number of reduced basis functions for the same initial tolerance, and achieves
higher speedup since the RB construction time does not change much. Furthermore, compared to the fixed RB construction with relatively small tolerance at the initial step, adaptive RB leads to higher speedup, while achieving higher accuracy than the former as shown in Fig.\ \ref{fig:accuracy-lognormal}. We note that the reduced basis (averaged) approximation (of $\eta_y^h$) errors are smaller than the HiFi (averaged) approximation errors of about $10^{-4}$ at SVGD step $l = 199$, even with the initial tolerance $\varepsilon_r^0 = 0.1$ as can be observed from Fig.\ \ref{fig:rom-update-lognormal} and in particular for $\varepsilon_r^0 = 0.01$ as seen from Fig.\ \ref{fig:accuracy-lognormal}. Therefore, the adaptive RB may achieve much higher speedup if the high-fidelity approximation is refined to achieve errors of about $10^{-8}$ as seen for the reduced basis approximation $\eta_r^\Delta$ in Fig.\ \ref{fig:accuracy-lognormal}.

\end{document}